\documentclass{article}%

\usepackage[T1]{fontenc}
\usepackage{dirtytalk}
\usepackage[breaklinks=true, colorlinks=true, linkcolor={blue!90!black}, citecolor={blue!90!black}, urlcolor={blue!90!black}]{hyperref}
\usepackage{comment}
\usepackage{url}
\usepackage{amsmath,amssymb,mathtools,doi}
\usepackage[linesnumbered,vlined,algoruled]{algorithm2e}
\usepackage{cleveref}

\DeclareMathOperator\sign{sign}

\newcommand{\Reals}{\mathbb{R}} %

\newcommand{\eps}{\varepsilon} 			%
\newcommand{\abs}[1]{\left\vert #1 \right\vert} %

\newcommand{\zerob}{\mathbf{0}} %
\newcommand{\oneb}{\mathbf{1}} %
\usepackage{bm} %

\newcommand{\Phib}{\bm{\Phi}}
\newcommand{\pib}{\bm{\pi}}

\newcommand{\etab}{\bm{\eta}}
\newcommand{\Omegab}{\bm{\Omega}}
\newcommand{\Gammab}{\bm{\Gamma}}
\newcommand{\sigmab}{\bm{\sigma}}
\newcommand{\ab}{\bm{a}}
\newcommand{\Ab}{\bm{A}}

\newcommand{\db}{\bm{d}}

\newcommand{\eb}{\bm{e}}

\newcommand{\lb}{\bm{l}}

\newcommand{\Mb}{\bm{M}}

\newcommand{\Rb}{\bm{R}} %

\newcommand{\Sb}{\bm{S}}
\newcommand{\Tb}{\bm{T}}
\newcommand{\ub}{\bm{u}}

\newcommand{\vb}{\bm{v}}

\newcommand{\xb}{\bm{x}}

\newcommand{\yb}{\bm{y}}

\newcommand{\zb}{\bm{z}}

\newcommand{\lambdab}{\bm{\lambda}}

\newcommand{\mub}{\bm{\mu}}
\usepackage{wrapfig}

\newcommand{\cP}{\mathcal{P}} %

\usepackage{xcolor}
\newenvironment{proof}{\paragraph{Proof:}}{\hfill$\square$}
\usepackage[maxnames=5,style=numeric-comp,firstinits=true,sorting=none,backend=bibtex]{biblatex}
\renewbibmacro{in:}{}
\usepackage[margin=1in]{geometry}
\usepackage{authblk}
\providecommand{\keywords}[1]{\noindent\textbf{ Keywords:} #1}

\newcommand{\boundary}{\operatorname{Bndry}}
\newcommand{\interior}{\operatorname{Int}}
\newcommand{\diag}{\operatorname{diag}}

\newtheorem{lemma}{Lemma}

\newcommand{\ft}{\widetilde{f}}

\definecolor{ForestGreen}{RGB}{34,139,34}

\newtheorem{assumption}{Assumption}

\newtheorem{theorem}{Theorem}%
\newtheorem{remark}{Remark}%

\newtheorem{definition}{Definition}%
\newtheorem{property}{Property}

\makeatletter
\def\namedlabel#1#2{\begingroup
    #2%
    \def\@currentlabel{#2}%
    \phantomsection\label{#1}\endgroup
}
\makeatother

\begin{document}

\title{Modeling Approaches for Addressing Simple Unrelaxable Constraints with %
Unconstrained Optimization Methods}

\author[1]{Jeffrey Larson}%

\author[1,2]{Misha Padidar}%

\author[1]{Stefan M. Wild}%

\affil[1]{Mathematics and Computer Science Division, Argonne National Laboratory \linebreak[4] \texttt{\small jmlarson@anl.gov}; \texttt{\small wild@anl.gov}}

\affil[2]{Center for Applied Mathematics, Cornell University\linebreak[4] \texttt{\small map454@cornell.edu}}

\maketitle

\abstract{
We explore novel approaches for solving nonlinear optimization problems with unrelaxable bound constraints, which must be satisfied before the objective function can be evaluated. Our method reformulates the unrelaxable bound-constrained problem as an unconstrained optimization problem that is amenable to existing unconstrained optimization methods. The reformulation relies on a domain warping to form a merit function; the choice of the warping determines the level of exactness with which the unconstrained problem can be used to find solutions to the bound-constrained problem, as well as key properties of the unconstrained formulation such as smoothness. We develop theory when the domain warping is a multioutput sigmoidal warping, and we explore the practical elements of applying unconstrained optimization methods to the formulation. We develop an algorithm that exploits the structure of the sigmoidal warping to guarantee that unconstrained optimization algorithms applied to the merit function will find a stationary point to the desired tolerance.}


\section{Introduction}
\label{sec:intro}
This paper addresses nonlinear bound-constrained optimization problems 
\begin{equation}
\underset{\yb\in \Omegab =[\lb,\ub]}{\min} 
\; 
f(\yb), 
\tag{PROB}
\label{eq:PROB}
\end{equation}
where $f$ is a differentiable scalar-valued objective function and $\yb$
is an $n$-dimensional vector of decision variables. The decision space (or ``feasible region'') $\Omegab$ is a key factor here because we assume
that the constraints defining $\Omegab$ are unrelaxable~\cite{taxonomy15} and therefore the objective function $f$ cannot be evaluated at points outside $\Omegab$. Such constraints arise in settings including those where numerical simulations, differentiable algebraic equations, and other complex systems are known to not produce meaningful output when certain unrelaxable constraints are violated. For example, negative concentration levels in a chemical system and negative probabilities of transmission in an epidemiological simulation are readily modeled as belonging to regions that an optimization algorithm should never probe.  Many approaches for constrained optimization (e.g., penalty and filter methods) do not natively support such unrelaxable constraints, making always-feasible algorithms (beyond interior-point and projection-based approaches) an active area of research~\cite{Hallock2021PhD,Hough2021}.

Here we explore the use of a \say{domain warping}, $\Phib:\Reals^n\to\Omegab$ (defined in \Cref{sec:merit_connection_maps}), to develop a merit function $\tilde{f}$ that alleviates the dependence on unrelaxable constraints so that solutions to the \emph{unconstrained} optimization problem
\begin{equation}
\underset{\xb\in \Reals^n}{\min} 
\; 
\ft(\xb) := f(\Phib(\xb))
\tag{WPROB}
\label{eq:WPROB}
\end{equation}
can be transformed into solutions to \ref{eq:PROB} through the domain warping. We adopt the term ``domain warping'' from the image processing community; see, for example,~\cite[Figure 3]{Stefanoski2013}. We leverage the domain warping to develop an analog of classical penalty approaches~\cite{nocedal2006no} that applies to problems with unrelaxable bound constraints. While various domain warpings are available, we focus on a multioutput sigmoidal warping. This formulation enjoys smoothness, is easy to use, can solve \ref{eq:PROB} accurately, and can leverage the vast suite of unconstrained optimization solvers, including those that can exploit specialized objective function forms. Furthermore, we develop an algorithm that exploits the structure of the sigmoidal warping to generate a sequence of solutions of \ref{eq:WPROB} whose warped limit points satisfy the Karush--Kuhn--Tucker (KKT) conditions for \ref{eq:PROB}. Under mild conditions we prove convergence as well as convergence rates of this algorithm when gradient descent is used as a subproblem solver.

In \Cref{sec:merit_connection_maps} we describe the fundamentals of using domain warpings to reformulate \ref{eq:PROB}. 
The choice of domain warping is critical because it affects the solution set of \ref{eq:WPROB} and thus the level of exactness with which solutions to \ref{eq:WPROB} can be transformed into solutions to \ref{eq:PROB}, as well as key properties of the unconstrained formulation such as smoothness; see \Cref{fig:connections_on_rosen}. We show in \Cref{sec:sigmoidal_connection} that the sigmoidal warping can be used to find interior solutions to \ref{eq:PROB} exactly and boundary solutions in a limiting sense. We also explore the practical elements of applying typical unconstrained optimization algorithms to \ref{eq:WPROB}.
After summarizing related work in \Cref{sec:related}, we present our numerical results in \Cref{sec:numerical_exp} to illustrate the performance on bound-constrained optimization problems. 
Additional algorithmic considerations based on the effect of the sigmoidal warping are described in Appendix~\ref{app:alg_recs}.

Throughout the paper, we employ the following core assumptions.
\begin{assumption}
\label{assumption:f_unrelaxable}
$f:\Omegab\to\Reals$ can be evaluated only at $\yb$ in $\Omegab$.
\end{assumption}

\begin{assumption}
\label{assumption:f_L_smooth}
$f$ is continuously differentiable on its domain, and the $i$th partial derivative of $f$ is $L_i$-Lipschitz continuous for  $i=1,\ldots,n$ ($f$ is $L$-smooth with $L = \sqrt{\sum{L_i^2}}$; see \Cref{def:L-smooth}). 
\end{assumption}

\begin{assumption}
\label{assumption:omega_cube}
Since any bound-constrained region with finite bounds satisfying $l_i < u_i$ for $i=1,\ldots,n$ can be rescaled to the unit cube, without loss of generality we assume that $\Omegab$ is the unit cube $[0,1]^n$.
\end{assumption}

\begin{figure}[t]
\includegraphics[scale=0.35]{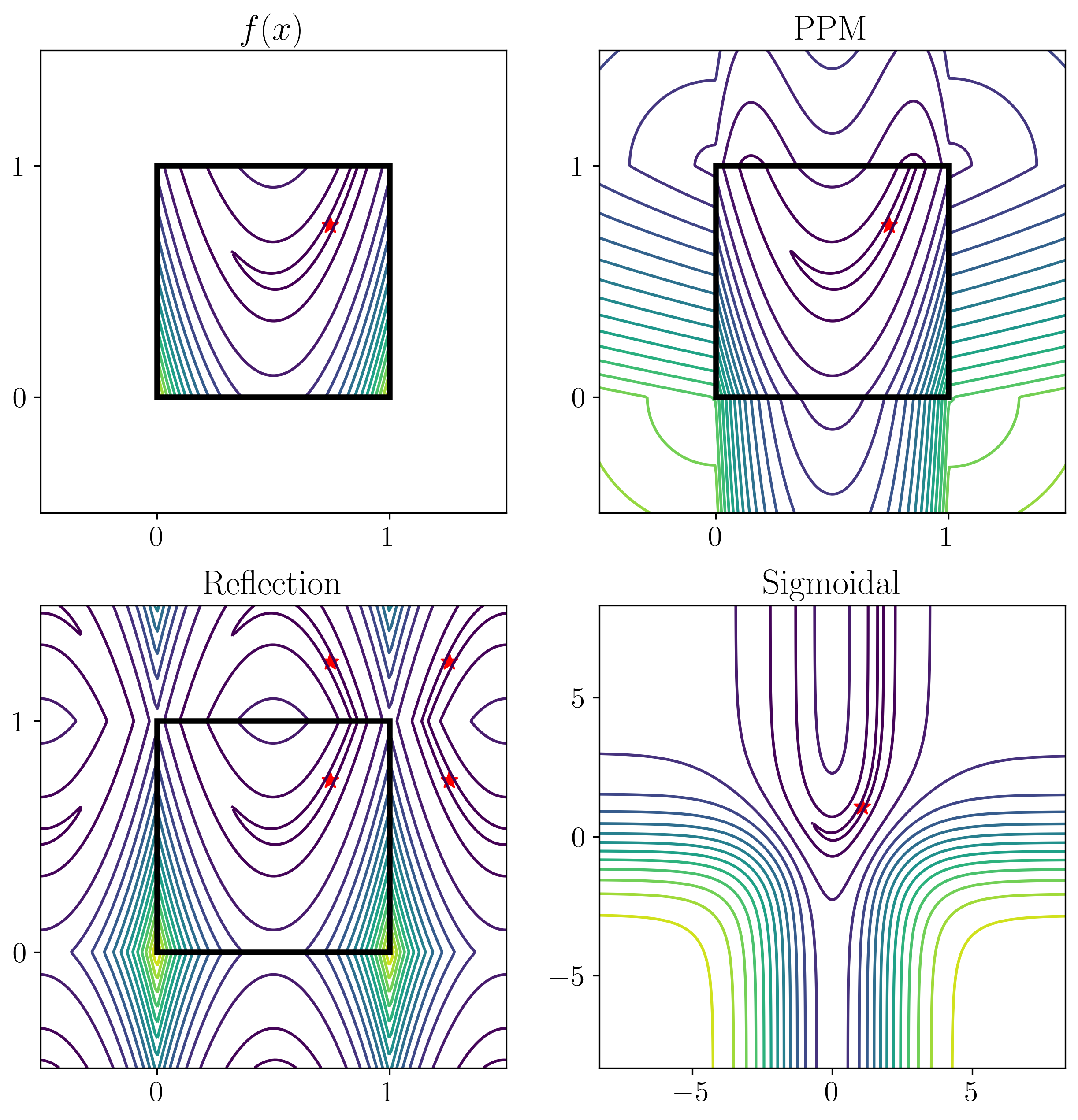}
\centering
\caption{Rosenbrock function mapped to the unit domain $[0,1]^2$ (top left) and merit functions under the sigmoidal, projection, and reflection domain warpings.
The sigmoidal warping forms the merit function in \cref{eq:ft_sigma_def}, the projection warping uses a distance penalty in the merit function \cref{eq:PPM},
and the reflection warping forms the merit function $f(\Rb(\xb))$ defined in \Cref{sec:merit_connection_maps}. The red star denotes the local minimum.}
\label{fig:connections_on_rosen}
\end{figure}

We use bold variables to indicate vectors and vector-valued mappings (e.g., $\xb$, $\Phib$). We use the componentwise product of vectors $\xb \odot\yb$ with entries $[\xb\odot\yb]_i=x_iy_i$ and componentwise quotients $\frac{\xb}{\yb}$ with entries $[\frac{\xb}{\yb}]_i=\frac{x_i}{y_i}$. We also compare vectors with inequalities: $\xb > \yb$ if and only if $x_i > y_i$ for $i=1,\ldots,n$. For sequences of vectors $\xb_k$, we index the components with double index notation $x_{k,i} = [\xb_{k}]_i$.
We use $\oneb$ to denote the vector of all ones and $\bm{e}_i$ to denote the vector of zeros with a one in component $i$. We denote the complement of an index set $I\subseteq\{1,\ldots,n\}$ by $I^c$.
The space $\Reals^n_{++}$ denotes the strictly positive orthant of $\Reals^n$. The norm $\|\cdot\|$ denotes the Euclidean norm.

\section{Merit Functions Based on Domain Warping}
\label{sec:merit_connection_maps}

By defining a continuous map, which we call a \textit{domain warping}, or simply \say{warping}, $\Phib:\Reals^n\to\Omegab$, we can construct a merit function $\ft(\xb) = f(\Phib(\xb))$ that can be used to find the minima of \ref{eq:PROB}, assuming a suitable warping is chosen. 
The primary focus of this paper is the vector-valued sigmoidal warping $\Sb(\xb) = \oneb/(\oneb+e^{-\sigmab\odot\xb})$ parameterized by $\sigmab\in\Reals^n_{++}$. The principal benefit of the sigmoidal warping is that it is smooth and thus the resulting merit function 
\begin{equation}
    \ft_{\sigmab}(\xb) = f(\Sb(\xb))
    \label{eq:ft_sigma_def}
\end{equation}
used in \ref{eq:WPROB} is smooth (when $f$ is smooth) and moreover the unconstrained problem can be solved by smooth optimization techniques. The sigmoidal warping benefits from being invertible and therefore also maps onto $\interior(\Omegab) = (0,1)^n$. While this implies that there is no point $\xb$ in the unconstrained domain that $\Sb$ maps to the boundary of $\Omegab$, the merit function $\ft_{\sigmab}$ can nonetheless be used to find boundary-lying KKT points of \ref{eq:PROB}: we show in \Cref{thm:limiting_stationarity} that as a sequence of points $\yb_k\in\interior(\Omegab)$ approaches a boundary KKT point of \ref{eq:PROB}, the corresponding sequence $\xb_k = \Sb^{-1}(\yb_k)$ in the unconstrained domain approaches a stationary point of the merit function. While the sigmoidal warping applies only to bound-constrained regions, analogs of the sigmoidal warping can be developed for nonnegativity constraints; simplexes; quadrilaterals; any smooth, invertible mappings of the former; and Cartesian products of these regions (see Appendix~\ref{appendix:warpings_for_other_sets}). 

Although we work with the sigmoidal warping,  %
similar
alternatives (e.g., $\tanh$, $\arctan$) are also smooth and invertible and map onto $(0,1)^n$. These maps can be characterized as smooth, componentwise, strictly increasing, bijective maps from $\Reals^n$ onto $\interior(\Omegab)$, possessing bounded derivatives.

In general, we do not need to restrict properties such as the smoothness, injectivity, or surjectivity of the domain warping; rather, we explore how these choices affect the solution set and practical elements of optimizing \ref{eq:WPROB}. For  bound-constrained problems we compare with two other warpings: the projection $\pib(\xb)$ onto the decision set and the \textit{reflection} warping $\Rb(\xb) = 2\abs{\frac{\xb}{2} - \lfloor \frac{\xb}{2} + \frac{\oneb}{2} \rfloor}$, otherwise known as a triangle wave with period two. \Cref{fig:connections_on_rosen} shows three merit functions for \ref{eq:PROB} when $f$ is the Rosenbrock function under the three choices of domain warping (sigmoidal, reflection, and projection). The projection and reflection warpings are nonsmooth,  noninvertible, but surjective warpings. Although the nonsmoothness gives these warpings access to the boundary of the domain, when used in a merit function, they require a penalty on the distance $d$ to the decision set (i.e., $f(\Phib(\xb)) + d(\xb)$) to ensure that all optima lie in $\Omegab$. Furthermore, the countably infinite non-differentiabilities generated by the reflection warping $\Rb$ pose practical problems for optimization. \Cref{fig:connection_maps} shows one-dimensional plots of the three domain warpings.
\begin{figure}[t]
\includegraphics[scale=0.35]{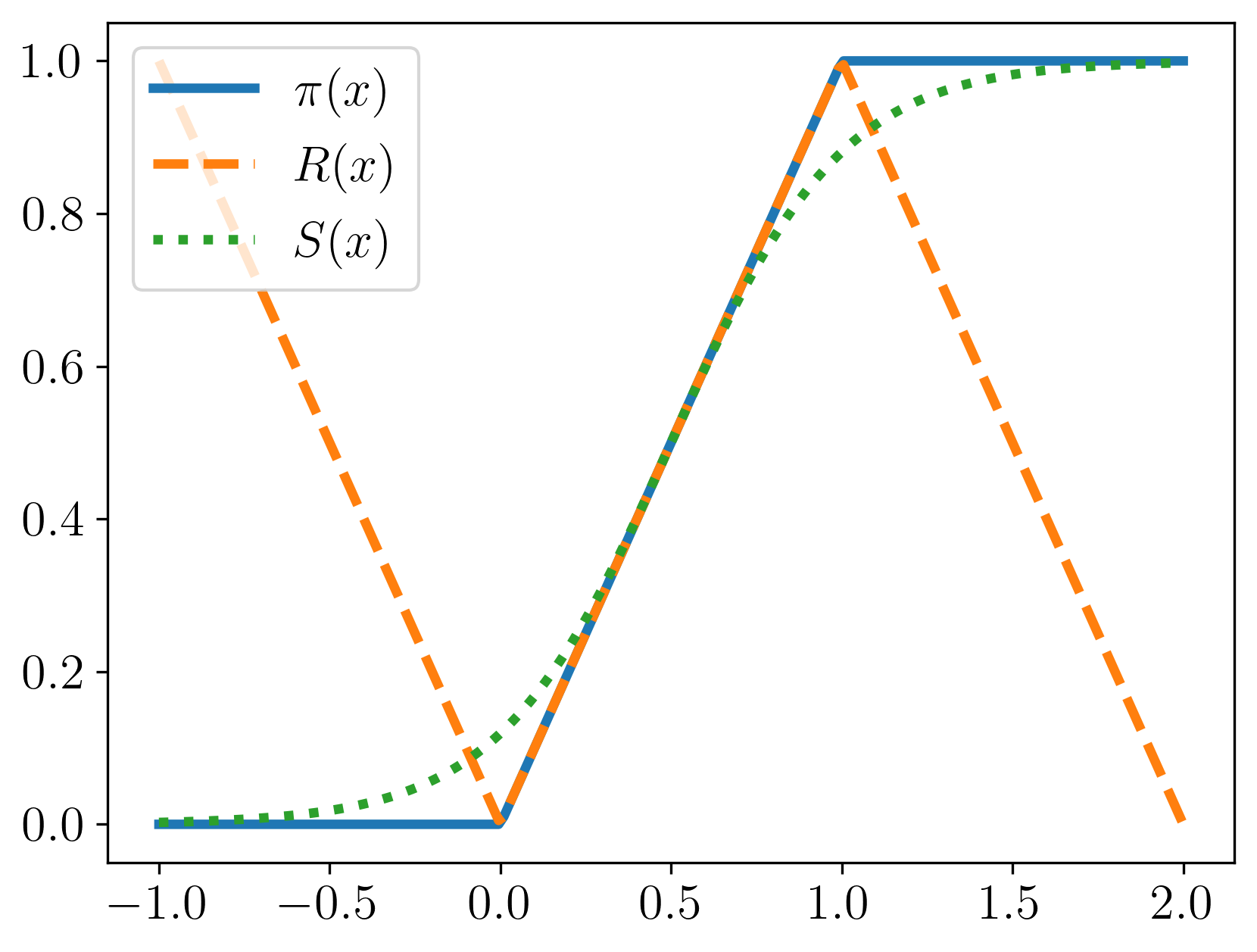}
\centering
\caption{Three domain warpings -- projection $\pib$, reflection $\Rb$, and sigmoidal $\Sb$ -- over the one-dimensional domain $[-1,2]$.}
\label{fig:connection_maps}
\end{figure}

Choosing $\Phib$ to be the projection warping and including a distance penalty
results in the projected penalty merit (PPM) function
from~\cite{Galvan2021}:
\begin{equation}
    \ft_{\pib}(\xb) = f(\pib(\xb)) + d(\xb).
    \tag{PPM}
    \label{eq:PPM}
\end{equation}
We numerically investigate the performance of this merit function in \Cref{sec:numerical_exp}.
Because this projection is constant along the normal cone
$N_{\Omegab}(\xb)$
of $\Omegab$ at a boundary point
$\xb$, a distance penalty is necessary to guarantee that the set of boundary-lying (Clarke) stationary points of the merit function are solutions to \ref{eq:PROB}. With the distance
penalty, the solutions to \ref{eq:WPROB} are exactly the solution set of
\ref{eq:PROB}, in that they share the same local minima and Clarke
stationary points~\cite{Galvan2021,Clarkebook}. From a practical standpoint,
using the projection warping can be cumbersome because the reformulation
is now nonsmooth along the boundary of $\Omegab$, making high-fidelity
resolution of minima along the boundary potentially difficult. A benefit of the projection
warping is that it can be readily used for any convex $\Omegab$, because the projection is uniquely
defined.

Similar to the projection warping, the reflection warping is an identity mapping for $\xb\in\Omegab$ and also induces nonsmoothness in the objective at the boundary of $\Omegab$. The reflection warping differs in that it is periodic; therefore, even, periodic objective functions such as $\cos(2\pi x)$ are repeated exactly on each interval $[k,k+1]$ for $k\in\mathbb{Z}$. In other words, when $f$ is even and periodic on $\Omegab$, the reflection warping tiles the unconstrained domain $\Reals$ with copies of \ref{eq:PROB}. In addition, the reflection warping applies only for particular choices of $\Omegab$ such as bound-constrained regions and half-planes. From a practical standpoint, the reflection warping $\Rb$ does not provide significant benefits over the projection warping, and so we do not discuss it further.

\section{Using the Sigmoidal Warping}
\label{sec:sigmoidal_connection}

We now detail the effects
of using the sigmoidal warping within the merit function $\ft_{\sigmab}$ in \ref{eq:WPROB}. In \Cref{sec:sigup} we propose an iterative scheme for updating the $\sigmab$ parameter, similar to a penalty method, that guarantees convergence of a sequence of solutions of \ref{eq:WPROB} to KKT points of \ref{eq:PROB} on the interior or boundary of $\Omegab$. In \Cref{sec:alg1_analysis} we prove the convergence of this method as well as bound the convergence rate, under mild assumptions.

First, we discuss basic properties of the sigmoidal warping. To map the unconstrained domain onto $(0, \, 1)^n$, we apply the one-dimensional sigmoidal warping $s(x)$ to each entry of the vector $\xb$, thus forming the vector-valued map $\Sb(\xb)$, where
\begin{equation}
    \Sb(\xb)_i = s(x_i) := \frac{1}{1+e^{-\sigma_ix_i}}.
    \label{eq:sigmoidal}
\end{equation}
The one-dimensional sigmoidal warping $s$ has derivatives
\begin{equation*}\label{eq:sigmoidal_derivatives}
    s'(x) = \sigma s(x)(1-s(x))
    \quad \mbox{and} \quad 
    s''(x) =\sigma^2 s(x)(1-s(x))(1-2s(x)),
\end{equation*}
and so the first- and second-order derivatives of the merit function $\ft_{\sigmab}$ are
\begin{align}
    \nabla \ft_{\sigmab}(\xb) &= J_{\sigmab}(\xb)\nabla f(\Sb(\xb))
    = \sigmab\odot\Sb(\xb)\odot(\oneb-\Sb(\xb))\odot \nabla f(\Sb(\xb))
    \\
    D^2\ft_{\sigmab}(\xb) &= H_{\sigmab}(\xb)\diag\left(\nabla f(\Sb(\xb))\right) + J_{\sigmab}(\xb)D^2f(\Sb(\xb)) J_{\sigmab}(\xb),
    \label{eq:ft_derivatives}
\end{align}
where the diagonal Jacobian $J_{\sigmab}(\xb)$ has diagonal $\sigmab\odot \Sb(\xb)\odot(\oneb-\Sb(\xb))$ and the diagonal second-derivative matrix $H_{\sigmab}(\xb)$ has diagonal $\sigmab^2\odot\Sb(\xb)\odot(\oneb-\Sb(\xb))\odot(\oneb-2\Sb(\xb))$. Importantly, the eigenvalues of the Jacobian are strictly positive when $\sigmab >\zerob$, making $J_{\sigmab}(\xb)$ positive definite for any $\xb$. Because the sigmoidal warping has Lipschitz-continuous derivatives and $\partial f$ is Lipschitz continuous, the derivatives of $\ft_{\sigmab}$ are also Lipschitz continuous, with a constant that depends on $\sigmab$.

\begin{lemma}
\label{thm:lipschitz_consant_bound}
If $f$ and $\nabla f$ are $\hat{L}$- and $L$-Lipschitz continuous, respectively, then $\nabla \ft_{\sigmab}$ is $\tilde{L}$-Lipschitz continuous for $\tilde{L} = \frac{1}{2}(\sigma_{\max}^2\hat{L} + \sigma_{\max}L)$ and $\sigma_{\max} = \max_i \sigma_i$.
\end{lemma}
\begin{proof}
For any $\xb,\ub\in\Reals^n$ with $\yb=\Sb(\xb)$ and $\vb=\Sb(\ub)$, 
\begin{align*}
    \|\nabla \ft_{\sigmab}(\xb) - \nabla \ft_{\sigmab}(\ub) \| & = \|\sigmab\odot\yb\odot(\oneb-\yb)\odot\nabla f(\yb) - \sigmab\odot\vb\odot(\oneb-\vb)\odot \nabla f(\vb) \|
    \\
    & \le \sigma_{\max}\|\yb\odot(\oneb-\yb) \odot\nabla f(\yb) - \vb\odot(\oneb-\vb)\odot\nabla f(\vb) \|
    \\
    & \le \sigma_{\max}\|(\yb\odot(\oneb-\yb) -\vb\odot(\oneb-\vb))\odot\nabla f(\yb)\| \\
    & \hspace{10pt} + \sigma_{\max}\|\vb\odot(\oneb-\vb)\odot (\nabla f(\yb) - \nabla f(\vb))\|.
\end{align*}
We further extend the bound by using $\hat{L}$-Lipschitz continuity of $f$ to bound $\nabla f(\yb)$, the upper bound $\vb\odot(\oneb-\vb) \le \frac{1}{4}$, and $L$-Lipschitz continuity of $\nabla f$:
\begin{align*}
    &\le \sigma_{\max}\hat{L}\|(\yb\odot(\oneb-\yb) -\vb\odot(\oneb-\vb))\| + \frac{\sigma_{\max}}{4}\| \nabla f(\yb) - \nabla f(\vb)\|
    \\
    &\le \sigma_{\max}\hat{L}\|(\yb\odot(\oneb-\yb) -\vb\odot(\oneb-\vb))\| + \frac{\sigma_{\max}L}{4}\|\xb-\ub\| .
\end{align*}
A Lipschitz bound of $\Sb(\xb)\odot(\oneb-\Sb(\xb))$ is $\sigma_{\max}/2$, which can be computed by upper bounding the eigenvalues of the Jacobian. We arrive at our final expression:
\begin{align*}
    &\le \frac{2\sigma_{\max}^2\hat{L} + \sigma_{\max}L}{4}\|\xb-\ub\|
    \le \frac{1}{2}\left(\sigma_{\max}^2\hat{L} + \sigma_{\max}L\right)\|\xb-\ub\| .
\end{align*}
\end{proof}

A downside of the sigmoidal warping is that it does not preserve convexity. 
Even if $f$ is (strongly) convex, $\ft_{\sigmab}$ may not be (strongly) convex. 
Despite not preserving convexity, this warping preserves interior minima and stationary points. The following theorem establishes a one-to-one mapping between interior stationary points of $f$ and stationary points of $\ft_{\sigmab}$.
\begin{theorem}
Let $f$ and $\Omegab$ satisfy \Cref{assumption:f_L_smooth} and \Cref{assumption:omega_cube} and let $\sigmab >\zerob$. A point $\yb^*$ on the interior of $\Omegab$ is a stationary point of $f$ if and only if its warped point $\xb^* = \Sb^{-1}(\yb^*)$ is a stationary point of $\ft_{\sigmab}$.
\label{thm:interior_stationary_points}
\end{theorem}

\begin{proof}
By \Cref{assumption:omega_cube} and \cref{eq:sigmoidal},
there is a one-to-one mapping between points $\yb^*$ in the interior of $\Omegab$ and $\xb^*$ in $\Reals^n$. By \Cref{assumption:f_L_smooth}, the first derivative of $\ft_{\sigmab}$ exists and is $\nabla \ft_{\sigmab}(\xb^*) = J_{\sigmab}(\xb^*)\nabla f(\yb^*)$. For any $\xb\in\Reals^n$ the null space of the Jacobian $J_{\sigmab}(\xb)$ is $\{\zerob\}$ since $\sigmab >\zerob$. Thus $\nabla \ft_{\sigmab}(\xb^*) = J_{\sigmab}(\xb^*)\nabla f(\yb^*) = \zerob$ if and only if $\nabla f(\yb^*) = \zerob$. 
\end{proof}

Unfortunately, $\ft_{\sigmab}$ cannot \textit{exactly} reproduce the behavior of $f$ at the boundary of $\Omegab$. The inverse domain warping $\Sb^{-1}:(0,1)^n \to \Reals^n$ is $\Sb^{-1}(\xb) = \log(\xb/(\oneb-\xb))/\sigmab$. Effectively, the boundary of $\Omegab$ maps to $\pm\infty$. In theory $\xb$ would have to become infinitely large for $\Sb(\xb)$ to recover a boundary point, including KKT points of \ref{eq:PROB}. In practice, however, $\xb$ need not become particularly large since $\xb$ can approximate boundary points of $\Omegab$ from the shrunken domain, $[a,1-a]^n$ for $0<a<1/2$, which maps to a finite interval under the domain warping.
We can measure the size of the effective domain of $\xb$ by bounding the norm of points $\xb\in\Reals^n$ when $\Sb(\xb)$ is a distance $a$ from the boundary.
\begin{remark}
Let $\yb\in[a,1-a]^n$ with $0< a <1/2$; then, $\left\| \Sb^{-1}(\yb) \right\|_\infty \leq \log\left(\dfrac{1-a }{a }\right)$.
\end{remark}
While it is necessary for an entry $x_{k,i}\to\pm\infty$ for $y_{k,i}$ to converge to a point on the $i$th boundary, the preceding remark implies that $x_{k,i}$ is bounded if $y_{k,i}$ remains a finite nonzero distance from the boundary. 
For instance, if $\Sb(\xb)$ is at least a distance $10^{-3}$ from a boundary point in infinity norm, then
$\xb$ has a largest entry of no more than $6.9$. 
Furthermore, if we restrict $\yb$ to be at least a distance of $10^{-16}$ from the boundary, we know that $\|\Sb^{-1}(\yb)\|_{\infty}< 37$.

Although our approach cannot exactly produce KKT points of $f$ on the boundary of $\Omegab$, the gradient of $\ft_{\sigmab}$ does approach zero as $\Sb(\xb)$ approaches a KKT point. In the following theorem we show that KKT points of \ref{eq:PROB} are limiting stationary points of the merit function $\ft_{\sigmab}$. The implication is that unconstrained optimization techniques that converge to stationary points of $\ft_{\sigmab}$ will be able to 
approximate KKT points of \ref{eq:PROB} arbitrarily well.
\begin{theorem}
\label{thm:limiting_stationarity}
Let $f$ and $\Omegab$ satisfy \Cref{assumption:f_L_smooth} and \Cref{assumption:omega_cube} and let $\sigmab >\zerob$. Let $\yb^*\in\Omegab$ be a KKT point of \ref{eq:PROB} and $\{\yb_k\}\in\interior(\Omegab)$ be a sequence of points that converge to $\yb^*$ with $\{\xb_k = \Sb^{-1}(\yb_k)\}$ being the corresponding points in $\Reals^n$. Then $\nabla \ft_{\sigmab}(\xb_k)\to 0$.
\end{theorem}
\begin{proof}
The case where $\yb^* \in \interior(\Omegab)$ is shown by a direct application of \Cref{thm:interior_stationary_points} to the sequence $\{\xb_k\}$. We therefore consider the case where there is a nonempty set of indices $I$ where $y^*_i$ for $i \in I$ are on the boundary of $\Omegab$. By \Cref{assumption:omega_cube}, $y^*_i\in\{0,1\}$ for all $i \in I$. The sequence $\{\xb_k\}$ then diverges for the components in $I$: $x_{k,i} \to \sign(2y^*_i - 1)\infty$ for $i\in I$. By the definition of $\Sb$, the partial derivative of the warping converges to zero in each of these components as $x_{k,i}\to\sign(2y^*_i - 1)\infty$. That is,  $s'(x_{k,i}) \to 0$ for $i\in I$. The remaining components of $\{\xb_k\}$ converge, by continuity of $s$, to $x_{k,i} \to s^{-1}(y^*_i)$ for $i\in I^c$. 

Now that we know the limits of the sequences $\{\xb_k\}$ and $\{\yb_k\}$, we will find the limits of the partial derivative sequences $\partial_i f(\yb_k)$, for which we appeal to the KKT conditions. Since $\yb^*$ is a KKT point, there exists a Lagrange multiplier $\lambdab^*$ such that the KKT conditions hold (see \Cref{thm:kkt} in Appendix~\ref{sec:background}) in particular, the Lagrange multiplier satisfies dual feasibility $\sign(\lambda_i^*) \in\{ \sign(2y^*_i - 1),0\}$ for $i\in I$, complementary slackness $\lambda_i^* = 0$ for $i\in I^c$, and the stationary condition $\partial_i f(\yb^*) + \lambda_i^*  = 0$ for all $i$. Because $\partial_i f$ is continuous by  \Cref{assumption:f_L_smooth}, the sequence of partial derivatives $\{\partial_i f(\yb_k)\}$ converges to $ \partial_i f(\yb^*)$. Furthermore, by the KKT conditions, $\partial_i f(\yb_k) \to \partial_i f(\yb^*) = -\lambda_i^*$ for all $i$.

Now that we have the limit of the sequence $\{\partial_i f(\yb_k)\}$, we can compute the limit of the sequence $\{\partial_i \ft_{\sigmab}(\xb_k)\}$.
First, consider the components $i\in I^c$. For these components, the partial derivative at the KKT point satisfies $\partial_i f(\yb^*) =-\lambda_i^*= 0$. Because $\partial_i f$ is continuous, $s'$ is continuous, and compositions and products of continuous functions are continuous, $\partial_i \ft_{\sigmab}$ is continuous. Thus the sequence of partial derivatives $\{\partial_i \ft_{\sigmab}(\xb_k)\}$ with $\partial_i \ft_{\sigmab}(\xb_k) = \sigma_iy_{k,i}(1-y_{k,i})\partial_i f(\yb_k)$ converges to $\sigma_iy_{i}^*(1-y_{i}^*)\partial_i f(\yb^*) = 0$ for $i\in I^c$. Moreover, for $i\in I^c$ the sequence $\partial_i\ft_{\sigmab}(\xb_k)  \to 0$. 

Now consider the components, $i\in I$, of the stationary point $\yb^*$ that lie on the boundary of $\Omegab$. We show that $\{\partial_i\ft_{\sigmab}(\xb_k)\}$ converges to zero. As $y_{k,i} \to y^*_i$, for any $\rho>0$ there exists an $N$ such that for all $k \geq N$, $\|\yb_k - \yb^*\|_{\infty} \leq \rho$. By Lipschitz continuity of $\partial_i f$ with constant $L_i$ (\Cref{assumption:f_L_smooth}), $\abs{\partial_i f(\yb_k)} \leq \abs{\partial_if(\yb^*)} + L_i\abs{y_{k,i} - y^*_i}$. Substituting the Lagrange multiplier gives us $\abs{\partial_i f(\yb_k)} \leq \abs{\lambda_i^*} + L_i\abs{y_{k,i} - y^*_i}$. With these, we can bound the value of $\partial_i\ft_{\sigmab}(\xb_k)$:
\begin{align*}
    \abs{\partial_i\ft_{\sigmab}(\xb_k)} %
    &= s'(x_{k,i})\abs{\partial_if(\yb_k)}
    \; \leq \; s'(x_{k,i})(\abs{\lambda_i^*} + L_i\abs{y_{k,i} - y^*_i})
    \\
    &=\sigma_iy_{k,i}\left(1-y_{k,i}\right)\left(\abs{\lambda_i^*} +L_i\abs{y_{k,i} - y^*_i}\right)
    \\
    &\leq \rho\sigma_i\left(\abs{\lambda_i^*} +L_i\rho\right).
\end{align*}
Thus, as $y_{k,i}$ approaches $y^*_i$ (i.e., $\rho\to 0$), the sequence $\{\partial_i\ft_{\sigmab}(\xb_k)\}$ approaches zero, implying that the sequence $\partial_i\ft_{\sigmab}(\xb_k) \to 0$. 
Thus, $\nabla\ft_{\sigmab}(\xb_k) \to 0$.%
\end{proof}

While \Cref{thm:limiting_stationarity} demonstrates that $\ft_{\sigmab}$ can be used to find boundary-lying KKT points of \ref{eq:PROB}, it is not yet clear how accurately we can resolve these KKT points by optimizing $\ft_{\sigmab}$. The following theorem bounds the error in the KKT stationary condition in terms of  $\abs{\partial\ft_{\sigmab}}$ and the parameter $\sigmab$. This bound shows that when approaching a minimum near or on the boundary of $\Omegab$, the partial derivatives of $\ft_{\sigmab}$ may approach zero prematurely. However, the bound also shows that the parameter $\sigmab$ can control the error in the stationary condition, similar to a penalty parameter in a penalty method.

\begin{theorem}
\label{thm:stationary_condition}
For $i=1,\ldots,n$, if $\sigmab >\zerob$ and the partial derivatives satisfy $\abs{\partial_i \ft_{\sigmab}(\xb)} \leq \delta_i$, then $\abs{\partial_i f(\yb)} \le \dfrac{\delta_i}{\sigma_i y_i(1-y_i)}$ at $\yb = \Sb(\xb)$. 

If additionally $\abs{\partial_i f(\yb)} > 0$ and $\yb^*\in\Omegab$ is a KKT point of \ref{eq:PROB}, under \Cref{assumption:f_L_smooth} and \Cref{assumption:omega_cube}, with corresponding Lagrange multiplier $\lambdab^*$, then, for entries of $\yb^*$ that lie on the boundary of $\Omegab$,  
$\abs{\partial_if(\yb) +\lambda_i^*} \leq\dfrac{L_i\delta_i}{\abs{\partial_i f(\yb)}\sigma_i\Delta_i}$, where $\Delta_i=\abs{1-y_i-y_i^*}$. 
\end{theorem}
\begin{proof}
We begin by proving the first bound stated in the theorem. Since $\sigmab >\zerob$, the partial derivative of $f$ at $\yb$ is bounded by
\begin{equation}
\abs{\partial_if(\yb)}
=\abs{\dfrac{\partial_i\ft_{\sigmab}(\xb)}{s'(x_i)}} 
\leq \dfrac{\delta_i}{s'(x_i)} 
= \dfrac{\delta_i}{\sigma_i y_i(1-y_i)}.
\label{eq:kkt_err_bound_interior}
\end{equation}
Now we prove the bound on the boundary components of the KKT point, namely, $y^*_i$.
We use the fact that for components $y^*_i$ that lie on the boundary of $\Omegab$, $\abs{y_i-y^*_i} = \dfrac{s'(x_i)}{\sigma_i\abs{1-y_i-y^*_i}}$, which can be verified by using $s'(x_i) = (1-y_i)y_i\sigma_i$, $y^*_i \in \{0,1\}$ (by \Cref{assumption:omega_cube}), and noting that $(1-y_i)y_i\neq 0$ since $\yb = \Sb(\xb)$ is obtained through the sigmoidal warping. Then, we can bound the value of the stationary condition using the Lipschitz continuity of $\partial_i f$ (\Cref{assumption:f_L_smooth}):
\begin{align}
\abs{\partial_if(\yb) +\lambda_i^*} 
= \abs{\partial_if(\yb) - \partial_i f(\yb^*)}
&\leq L_i\abs{y_i - y^*_i}
= \dfrac{L_is'(x_i)}{\sigma_i\abs{1-y_i-y^*_i}} \notag
\\
&=\dfrac{L_i\partial_i\ft_{\sigmab}(\xb)}{\partial_i f(\yb)\sigma_i\abs{1-y_i-y^*_i}} \notag
\\
&\leq\dfrac{L_i\delta_i}{\abs{\partial_i f(\yb)}\sigma_i\Delta_i},
\label{eq:kkt_err_bound_boundary}
\end{align}
using the definition of $\Delta_i$.
\end{proof}

The bound on the stationary condition, \cref{eq:kkt_err_bound_interior}, shows that decreasing the value of $\abs{\partial_i\ft_{\sigmab}(\xb)}$ can decrease the value of $\abs{\partial_i f(\yb)}$. However, it also shows that as $y_i$ approaches the boundary of $\Omegab$, the bound on $\abs{\partial_i f(\yb)}$ increases, as the denominator $y_i(1-y_i)$ will approach zero.
For example, if the goal of the optimization is to find an approximate stationary point to within a gradient tolerance of $\abs{\nabla f(\yb)}\le \delta$, but the stationary point $\yb^*$ is a distance of $0.1$ from the boundary, then \cref{eq:kkt_err_bound_interior} states that optimizing the objective $\ft_{\sigmab}$ to a tolerance of $\delta$ will result in a stationary condition satisfaction no worse than $12\frac{\delta}{\sigma}$.
Therefore, in order to satisfy the stationary condition to a tolerance $\delta$ for interior stationary points, $\sigma$ should be increased to at least $12$.
For a fixed value of $\delta$, increasing $\sigmab$ improves the first error bound. 

In \Cref{thm:stationary_condition} we see an analogous effect of $\sigmab$ on the error in the stationary condition relative to a \textit{boundary-lying} KKT point of \ref{eq:PROB}. \Cref{eq:kkt_err_bound_boundary} shows that the violation of the stationary condition $\abs{\partial_i f(\yb) + \lambda^*_i}$ relative to a KKT point $\yb^*$ is increased if $\abs{\partial_i f(\yb)}$ becomes small when $y_i$ approaches a boundary-lying entry of the KKT point. Again, however, if increased sufficiently, $\sigmab$ can counter these adverse effects.

Particularly when a minimum lies on or near the boundary of $\Omegab$ it is useful to increase $\sigmab$ to improve approximate satisfaction of the KKT conditions. Below we develop an algorithm that finds KKT points of \ref{eq:PROB} to a desired tolerance by iteratively increasing $\sigmab$.

\subsection{Iterative Updating of $\sigma$}
\label{sec:sigup}

While the sigmoidal warping does not map onto the boundary of $\Omegab$, the merit function $\ft_{\sigmab}$ can still be used to approximate any stationary point of \ref{eq:PROB} arbitrarily well. As \Cref{thm:stationary_condition} highlights, however,  resolving stationary points can be challenging because optimizing $\ft_{\sigmab}$ to a gradient tolerance $\delta$ will not necessarily find a point $\xb$ satisfying $\abs{\partial f(\Sb(\xb))}\le \delta$ or analogous KKT satisfaction. The gradient $\nabla \ft_{\sigmab}(\xb) = \sigmab \odot\Sb(\xb)\odot(\oneb-\Sb(\xb))\odot\nabla f(\xb)$ will become small when either the gradient of $f$ becomes small or $\Sb(\xb)$ approaches the boundary of $\Omegab$. \Cref{thm:stationary_condition} not only elucidates this behavior but also shows that the parameter $\sigmab$ can be used as a control to improve the satisfaction of the KKT conditions of \ref{eq:PROB}. Similar to the increase of the penalty parameter in penalty methods, convergence to a stationary point can benefit from iteratively increasing $\sigmab$. In \Cref{alg:sigma_update} we introduce a framework for solving \ref{eq:WPROB} that adaptively increases the entries of $\sigmab$ to guarantee convergence of a sequence of approximate stationary points of $\ft_{\sigmab}$ to the warping of a KKT point of \ref{eq:PROB}, whether on the interior or boundary of $\Omegab$. 
A key benefit of using this framework is that any unconstrained optimization routine can be used as a subproblem solver to find approximate stationary points of $\ft_{\sigmab}$; a convergence proof of \Cref{alg:sigma_update} when using gradient descent is provided in \Cref{thm:sigup_converges}. 

At the core of \Cref{alg:sigma_update} is an update rule, which defines a sequence $\{\sigmab_k\}$ that controls the domain warping, affects the sequence of iterates $\{\xb_k^*\}$ taken to the KKT point, and is ultimately responsible for the convergence properties of the algorithm. At each iteration, \Cref{alg:sigma_update} minimizes $\ft_{\sigmab_k}$ to a preset tolerance $\delta$ to find a point $\xb_k^*$ and warped point $\yb_k^* = \Sb(\xb_k^*)$, before increasing $\sigmab_k$ and reoptimizing. \Cref{thm:sigup_converges} shows that the only requirements to guarantee convergence when using gradient descent as a subproblem solver are for $\sigmab_k\to\infty$ and for the ratio of the smallest to the largest value of $\sigmab$ to be bounded below: $\min_{i,j}\{\sigma_{k,i},\sigma_{k,j}\}>\kappa$ for some fixed $\kappa>0$. Within these flexible requirements, the choice of update rule can significantly affect the rate of convergence. We later explore one update rule, \ref{eq:update_rule}, that achieves good practical performance when used in \Cref{alg:sigma_update}, and we bound the number of iterations of \Cref{alg:sigma_update} when using \ref{eq:update_rule} in \Cref{thm:sigup_num_iter}.

      \begin{algorithm}[H]
\footnotesize
\SetAlgoNlRelativeSize{-4}
\caption{Adaptive warping (AdaWarp)}
\label{alg:sigma_update}
\KwIn{Interior point $\yb_0 \in \interior(\Omegab)$, $\sigmab_0>\zerob$, tolerance $\delta>0$}

\KwResult{Approximate stationary point $\yb^*_k$}

 \For{$k=0,1,\ldots$}{
  Compute starting point $\xb_k = \Sb^{-1}(\yb_k)$ using $\sigmab=\sigmab_k$\;
  
  Obtain $\xb^*_k$ by approximately solving $\min_{\xb} \ft_{\sigmab_k}(\xb)$, starting from $\xb_k$, to a gradient tolerance $\|\nabla\ft_{\sigmab_k}(\xb_k^*)\| \le \delta$, and with $\ft_{\sigmab_k}(\xb_k^*)\le \ft_{\sigmab_k}(\xb_k)$\;
  
  Compute $\yb^*_k = \Sb(\xb^*_k)$ using $\sigmab=\sigmab_k$\;

  Select parameter $\sigmab_{k+1} > \sigmab_k$
  
  Update $\yb_{k+1} = \yb^*_k$\;
 }
\end{algorithm}

First note that by continuity of $f$ and $s$, and invertibility of $s$, there always exists a point $\xb_k^*$ that satisfies Step 3 of \Cref{alg:sigma_update}: that is, a point satisfying a gradient tolerance $\|\nabla\ft_{\sigmab_k}(\xb_k^*)\| \le \delta$ and nonincreasing function value $\ft_{\sigmab_k}(\xb_k^*)\le \ft_{\sigmab_k}(\xb_k)$. Explicitly, by \Cref{assumption:f_L_smooth} and \Cref{assumption:omega_cube} there is a KKT point $\yb^*$ of \ref{eq:PROB} with $f(\yb^*) \le f(\Sb(\xb_k))$. By continuity of $f$ a ball exists around $\yb^*$ such that all points within the ball and $\Omegab$ satisfy the requirements of Step 3 of \Cref{alg:sigma_update}. Thus, Step 3 of \Cref{alg:sigma_update} is viable. Moreover, the conditions in Step 3 can be satisfied by many gradient-based optimization routines.

\subsection{Convergence and Complexity of \Cref{alg:sigma_update}}
\label{sec:alg1_analysis}

In \Cref{thm:sigup_converges} we prove that every limit point of \Cref{alg:sigma_update} is indeed a KKT point of \ref{eq:PROB} if gradient descent is used as a subproblem solver, $\{\sigmab_k\}$ diverges to $\infty$, and the ratio of $\sigma_{k,i}/\sigma_{k,j}$ stays bounded for all $i,j,k$. Simple updates rules such as an exponential increase satisfy these requirements. However, rules that adaptively update the entries of $\sigmab_k$ based on the iterates $\xb_k^*$, such as \ref{eq:update_rule}, have the potential to perform much better because they can balance the conditioning of the objective in orthogonal directions based on the magnitude of $x^*_{k,j}$. We analyze the convergence rate of \Cref{alg:sigma_update} under \ref{eq:update_rule} in \Cref{thm:sigup_num_iter}. 

\begin{theorem}
\label{thm:sigup_converges}
Let $f$ and $\Omegab$ satisfy \Cref{assumption:f_L_smooth} and \Cref{assumption:omega_cube}. Let $\{\sigmab_k>\zerob\}\to\infty$ and let the ratio of the smallest to the largest value of $\sigmab_k$ be bounded below, that is, $\min_{i,j}\{\sigma_{k,i}/\sigma_{k,j}\} \ge \kappa>0$ for all $k$. Then, all limit points of \Cref{alg:sigma_update} are KKT points of \ref{eq:PROB} if gradient descent is used as a subproblem solver with a constant step size $\alpha_k = 1/\tilde{L}_k$, where $\tilde{L}_k$ is the Lipschitz constant of $\nabla \ft_{\sigmab_k}$.
\end{theorem}
\begin{proof}
  See Appendix~\ref{sec:appendix_proofs}.
\end{proof}

Although the statement of \Cref{thm:sigup_converges} holds only for gradient descent with a constant step size, it can readily be extended to the case with an adaptive step size. Furthermore, in the numerical experiments in \Cref{sec:numerical_exp}, we use L-BFGS in place of gradient descent because of its improved practical performance. 
\Cref{thm:sigup_converges} provides conditions under which \Cref{alg:sigma_update} generates subsequences that converge to a KKT point of \ref{eq:PROB}. The rate of convergence, however, will depend on the rule for increasing $\sigmab_k$.

We propose a rule that adaptively updates entries of $\sigmab_k$ based on the array of distances $\etab_k$ from the current iterate $\yb_k^*$ to $\boundary(\Omegab)$; that is, $\eta_{k,i} = \min\{y^*_{k,i},1-y^*_{k,i}\}$. Our rule also employs an exponential increase parameter $\gamma \ge 1$ that ensures that all components of $\sigmab$ increase at a baseline desired rate. Moreover, the rule ensures that the ratio of the smallest to the largest elements of $\sigmab_{k+1}$ is lower bounded by a preset value of $\kappa$. Entries of $\sigmab_k$ increase by the rule $\sigmab_{k+1} = \frac{\gamma}{\sqrt{\etab_{k}}}\odot\sigmab_k$ unless this value violates the lower bound on the ratio of the smallest to the largest elements of $\sigmab_{k+1}$. In this case, entries of $\sigmab_{k+1}$ that violate the ratio are updated by $\kappa^{-1}\min_{i}\{\frac{\gamma}{\sqrt{\eta_{k,i}}}\sigma_{k,i}\}$. Explicitly, under the update rule, $\sigmab_k$ is increased at each iteration by
\begin{equation}
    \sigma_{k+1,j} =
    \begin{cases}
        \dfrac{\gamma}{\sqrt{\eta_{k,j}}} \sigma_{k,j}  & \text{if } \dfrac{\sigma_{k,j}}{\sqrt{\eta_{k,j}}}  \le \kappa^{-1}\min_{i}\{\frac{\sigma_{k,i}}{\sqrt{\eta_{k,i}}}\}
        \\
        \kappa^{-1}\min_{i}\{\frac{\gamma}{\sqrt{\eta_{k,i}}}\sigma_{k,i}\}  & \text{otherwise}.
    \end{cases}
    \tag{UPRULE}
    \label{eq:update_rule}
\end{equation}
To ensure that limit points of \Cref{alg:sigma_update} are KKT points of \ref{eq:PROB}, \ref{eq:update_rule} guarantees that the ratio of the smallest to the largest entry of $\sigmab_k$ be bounded below by $\kappa>0$. This leads to a complicated expression to define the rule. If \Cref{alg:sigma_update} is set to terminate after reaching a finite stopping tolerance, such as an $\epsilon$-stationary tolerance for some fixed $\epsilon >0$ (see \cref{def:eps_stationary}), then the bound on the ratio of the smallest to the largest entries of $\sigmab_k$ will be implicitly enforced, allowing \ref{eq:update_rule} to be simplified. In order for \Cref{alg:sigma_update} to reach a finite $\epsilon$-stationary tolerance, $\sigmab_k$ does not need to diverge, and so the ratio of the smallest to the largest entry of $\sigmab_k$ will always be bounded below by some $\kappa$. Moreover, in this scenario there will exist $\kappa$ such that $\dfrac{\sigma_{k,j}}{\sqrt{\eta_{k,j}}}  \le \kappa^{-1}\min_{i}\{\frac{\sigma_{k,i}}{\sqrt{\eta_{k,i}}}\}$ is always satisfied during the course of \Cref{alg:sigma_update}, simplifying \ref{eq:update_rule} to $\sigmab_{k+1} = \frac{\gamma}{\sqrt{\etab_{k}}}\odot\sigmab_k$. In the numerical experiments, we use precisely this simplified update rule.

Other simple update rules, such as an exponential increase $\sigmab_{k+1} = \gamma \sigmab_k$, could be used in place of \ref{eq:update_rule}. However, we find that the proposed rule performs well in practice because of its ability to stabilize the ill-conditioning that appears as some entries of $\yb_k^*$ approach the boundary and others remain interior. 

While \ref{eq:update_rule} guarantees under \Cref{thm:sigup_converges} that limit points of \Cref{alg:sigma_update} are KKT points of \ref{eq:PROB}, we have yet to develop a bound on the number of iterations under this update rule to converge to a desired stopping tolerance. To arrive at this result, in \Cref{thm:sigup_num_iter}, we make the following mild assumption about the performance of \Cref{alg:sigma_update} under the update rule.

\begin{assumption}
There exists $\xi \in (0,1)$ so that $\Delta_{k,j} = \min_j\abs{1-y_{k,j}^* - z_j^*} \geq \xi$ for any limit point $\zb^*$ on the boundary of $\Omegab$ produced by the sequence of iterates $\{\yb_k^*\}$ from \Cref{alg:sigma_update} using \ref{eq:update_rule} to update $\sigmab_k$. 
\label{assumption:same_side}
\end{assumption}

The assumption states that the sequence produced by \Cref{alg:sigma_update} under \ref{eq:update_rule} can never approach the boundary opposite limit points of the sequence. This mandates that limit points of the sequence cannot be on opposing boundaries, and it bounds the distance between limit points. This assumption is relatively mild, since in practice we expect that the sequence produced by the algorithm will cluster near a single stationary point in the domain.

\Cref{assumption:same_side} provides us with the foundation to count the number of iterations required for \Cref{alg:sigma_update} to reach a desired stopping criterion. For this bound-constrained optimization problem, a reasonable stopping criterion for the solution $\yb_k^*$ of \Cref{alg:sigma_update} is $\epsilon$-stationarity (\Cref{def:eps_stationary} in Appendix~\ref{sec:background}). Alternatively, another common measure, the norm of the projected gradient $\|\pib(\yb_k^*-\nabla f(\yb_k^*))-\yb_k^*\| $, could be used. In \Cref{thm:sigup_num_iter} we compute a bound on the number of iterations required for \Cref{alg:sigma_update} to terminate to a desired $\epsilon$-stationarity tolerance.

\begin{theorem}
\label{thm:sigup_num_iter}
Let tolerances $\epsilon,\delta>0$ be given, along with \ref{eq:update_rule}'s increase parameter $\gamma \ge 1$, and $\sigmab_0 > 0$, and let $\yb_0$ be a feasible initial point to \ref{eq:PROB} under Assumptions~\ref{assumption:f_L_smooth} and \ref{assumption:omega_cube}. Let $\overline{L} = \max_{j}\{L_j\}$ be the largest Lipschitz constant of $\partial_j f$, and let $\xi$ be defined as in \Cref{assumption:same_side}.

Let Assumption~\ref{assumption:same_side} hold, and suppose that \Cref{alg:sigma_update} is run with the update rule \ref{eq:update_rule} and gradient descent as a subproblem solver with a $1/\tilde{L}_k$ step size at the $k$th iteration of \Cref{alg:sigma_update}, where $\tilde{L}_k$ is the Lipschitz constant of $\nabla \ft_{\sigmab_k}$. If there exist $\nu >0$ such that the sequence $\{y_{k,j}^*\}$ stays uniformly bounded at least a distance $\nu$ from the boundary in at least one component $j$, \Cref{alg:sigma_update} will find an $\epsilon$-stationary point in at most
\begin{equation*}
    N = \max\left\{\log\left(\frac{\delta}{\epsilon \nu (1-\nu )}\right)/\log(\sqrt{2}\gamma), \log\left(\frac{\overline{L}\delta}{\xi \epsilon^2} \right)/\log(\sqrt{2}\gamma)\right\}
\end{equation*}
iterations. If a $\nu$ exists for all components $j$, then the bound reduces to
\begin{equation*}
    N = \log\left(\frac{\delta}{\epsilon \nu (1-\nu )}\right)/\log(\sqrt{2}\gamma) .
\end{equation*}
If no $\nu$ exists, then \Cref{alg:sigma_update} finds an $\epsilon$-stationary point within $N$ iterations, where
\begin{equation*}
    N =  \log\left(\frac{\overline{L}\delta}{\xi \epsilon^2} \right)/\log(\sqrt{2}\gamma).
\end{equation*}
\end{theorem}

\begin{proof}
  See Appendix~\ref{sec:appendix_proofs}.
\end{proof}

While \Cref{thm:sigup_num_iter} bounds the number of iterations to $\epsilon$-stationarity, this bound can at times be uninformative because of the limits of finite computational precision. For instance, when converging to an interior stationary point at a distance $\nu = 10^{-8}$ from the boundary, with $\gamma=1$ and $\delta = \eps$, the bound shows that the algorithm must take at most $N=54$ iterations, long before which the $\sigmab$ parameter would have reached a numerical overflow. We find that this is not an issue because \Cref{alg:sigma_update} can typically achieve high orders of accuracy much sooner than the bound predicts. \Cref{thm:sigup_num_iter} not only bounds the number of iterations required to reach a desired KKT violation but also gives insight into how to select the $\delta$ parameter that governs the accuracy of the subproblem solve: a larger $\delta$ results in more iterations. To cancel out the effects of $\delta$ and $\epsilon$ in the bound, we can set $\delta = \epsilon$ for the ``interior-minimum'' case or $\delta = \epsilon^2$ when converging to a minimum on the boundary.

Empirically we find that the choice of $\sigmab_0$ greatly affects the path taken by \Cref{alg:sigma_update} and its rate of convergence.
\begin{figure}[tbh!]
\includegraphics[scale=0.3]{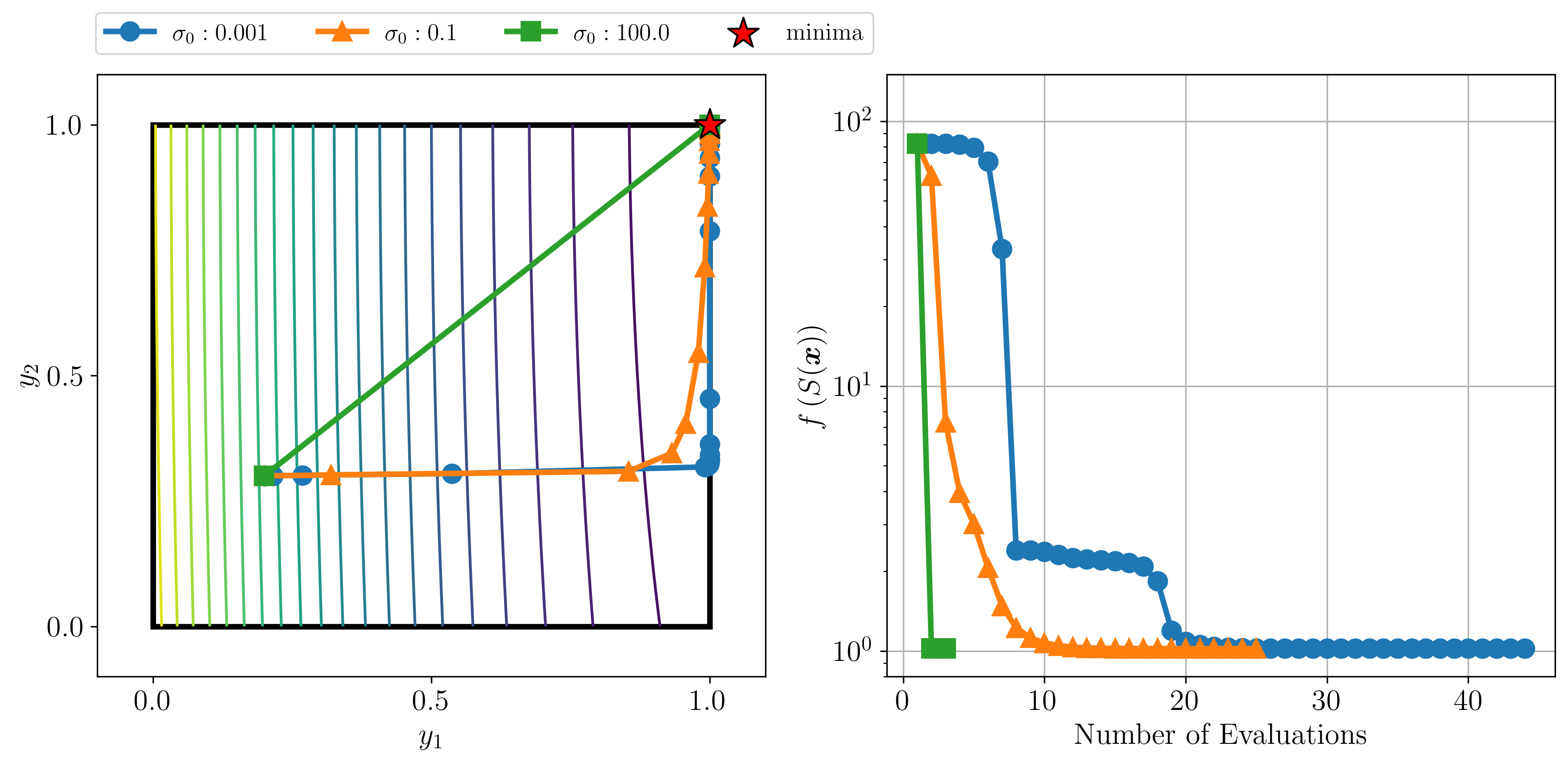}
\centering
\caption{Effect of $\sigmab_0$ on \Cref{alg:sigma_update} when minimizing $f$, an axis-aligned convex quadratic with minima at $(1.1,1.1)$ and Hessian eigenvalues $100,2$, over the unit cube. (Left) The three choices of $\sigmab_0$ lead to starkly different paths across the constrained domain toward the constrained optima at $(1,1)$. Approaching the boundary prematurely, such as by $\sigmab_0 = 0.001$, slows convergence of the method and requires many function evaluations to converge. On the other hand, motion along a \textit{central path}, such as the path taken by $\sigmab_0=100$, approaches the minima efficiently. (Right) The corresponding objective value over the three optimizations. All runs were terminated when the norm of the gradient was less than $10^{-6}$. }
\label{fig:effect_of_sigma0}
\end{figure}
Iterates that rapidly approach the boundary far from a boundary minimum may converge slowly, whereas iterations that take a more \textit{central path} tend to preserve a fast rate of approach; see \Cref{fig:effect_of_sigma0}. Mathematically, it is not trivial to predict the value of $\sigmab_0$ that will lead to the best convergence. In practice we suggest testing a few orders of magnitude such as $\{10^{-3},10^{-2},10^{-1},1,10\}$. In the numerical experiments, we find that a small $\sigmab_0=10^{-3}$ value tends to perform well on a CUTEst test problem set~\cite{Gould2014} when using \Cref{alg:sigma_update} and that $\sigmab=\oneb$ performs well when $\sigmab$ is not updated at all. As discussed in Appendix~\ref{sec:sigmoidal_gradient_steps}, setting $\sigmab_0 = \left(\Sb(\xb)\odot(\oneb-\Sb(\xb))\right)^{-1}$ will nullify the effect that the sigmoidal warping has on the first gradient step. While this is a principled selection, in practice we find that using a well-chosen constant $\sigmab_0$ is often superior. 

Additional algorithmic considerations are discussed in Appendix~\ref{app:alg_recs}.

\section{Related Optimization Work}
\label{sec:related}

Because of the simple nature of bound-constrained decision sets, many optimization routines~\cite{zhu1997LBFGSB,Powell2009a,Bertsekas1982,Planteng2009,LeDigabel:2011} have been adapted to handle the constraints directly by projecting iterates onto the decision set. The projection operation does occasionally posses other names such as \say{snap to boundary} \cite{LeDigabel:2011}. One such state-of-the-art routine, L-BFGS-B, leverages an approximate second-order model to rapidly resolve minima. We compare \Cref{alg:sigma_update} against an implementation of L-BFGS-B in the numerical experiments, with favorable performance. A sophisticated class of second-order interior-point methods~\cite{polik2010IPM,renegar2001IPM,nocedal2006no,wachter2006iip} takes the algorithmic developments a step further by leveraging a barrier function to ensure that iterates of the algorithm remain strictly on the interior of the domain. 

Penalty function-based methods~\cite{nocedal2006no,Fletcher} have also been used to solve nonlinear constrained optimization problems. The popular quadratic penalty approach allows for constrained optimization problems to be reformulated as a sequence of smooth unconstrained optimization problems, the solutions of which can converge to solutions of the original constrained problem. The quadratic penalty approach, however, suffers from poor conditioning because of the requirement that the penalty parameter be increased to infinity in order to ensure convergence to a feasible point. To ameliorate this, one can use an exact penalty approach or augmented Lagrangian methods~\cite{nocedal2006no,Bertsekas1982}, which apply a penalty term to the Lagrangian rather than the objective. Including an explicit estimate of the Lagrange multiplier improves convergence and conditioning of the sequence of unconstrained problems.

Penalty and augmented Lagrangian methods, however, cannot solve optimization problems with general unrelaxable constraints. Penalty and augmented Lagrangian methods leverage the fact that the objective can be evaluated outside the feasible region, and thus they cannot solve problems with unrelaxable constraints. Some augmented Lagrangian methods \cite{conn2013lancelot}, however, can solve bound-constrained problems through the use of specialized solvers because they explicitly solve bound-constrained subproblems rather than include bound constraints in the Lagrangian. 

Recent work has developed algorithmic approaches for handling unrelaxable constraints by the extreme barrier approach~\cite{Gratton2014,Audet2009}, including mesh adaptive direct search algorithms~\cite{Audet06mads}. A further algorithmic approach by Hough and Roberts~\cite{Hough2021} leverages the projection operator to handle unrelaxable constraints within a trust-region method. A thorough discussion of trust-region algorithms for problems with unrelaxable constraints is given in~\cite{Hallock2021PhD}.

Recently, Galvan et al.~\cite{Galvan2021} developed a general theory for projection-based penalty functions applied to optimization problems with unrelaxable convex decision sets. The projection-based penalty approach formulates a merit function $\ft_{\pib}$ using the projection $\pib(\xb)$ of $\xb$ onto $\Omegab$ and the distance $d(\xb)$ to $\Omegab$, \ref{eq:PPM}, which can be used to exactly find the Clarke stationary points~\cite{Galvan2021, Clarkebook} of the original problem. While PPM enjoys wide applicability, the composition with the projection operator induces nonsmoothness in the merit function, which poses a practical problem for numerical optimization routines. In \Cref{sec:merit_connection_maps} we looked at the projection operator as one specific domain warping, and in \Cref{sec:numerical_exp} we numerically investigate the performance of this method.

\section{Numerical Experiments}
\label{sec:numerical_exp}

We now compare the performance of \Cref{alg:sigma_update} against a state-of-the-art bound-constrained solver L-BFGS-B~\cite{zhu1997LBFGSB}, to approximate an upper bound on the performance of \Cref{alg:sigma_update}, and against the projected penalty approach from~\cite{Galvan2021}, to benchmark performance against a nonsmooth penalty reformulation of \ref{eq:PROB}. To be clear, our goal is not to show that \Cref{alg:sigma_update} is a dramatic improvement over state-of-the-art bound-constrained solvers, such as L-BFGS-B, but that a straightforward method running on an unfailingly simple modification of the objective can be used to solve optimization problems with unrelaxable bound constraints to a high accuracy and with reasonable efficiency. The code and data generated during the current study are available from the corresponding author on request.

We use data profiles~\cite{JJMSMW09} to measure the fraction of problems that an algorithm can solve to a given accuracy in a given number of function evaluations. For a problem set $\cP$ and a tolerance $\tau$, the data profile for an algorithm $a$ is 
\begin{equation*}
    d_a(\alpha) = \frac{1}{\abs{P}}\text{size}\left\{p\in P \,\bigg\vert\, \frac{t_{p,a}}{n+1}\le \alpha \right\},
\end{equation*}
where $t_{p,a}$ is the number of function evaluations taken by the algorithm to solve problem $p\in \cP$ to accuracy $\tau$.

We measure ``solving a problem'' in terms of relative KKT tolerance, which we define to be the value of the stationary condition at the given point, normalized to the norm of the objective's gradient at the nominal starting point $\yb_0$. That is, $\yb$ satisfies a relative KKT tolerance of $\tau$ if $\yb$ is $\epsilon$-stationary and $\frac{\epsilon}{\|\nabla f(\yb_0)\|} \leq \tau$.
Thus, $t_{p,a}$ is the first iteration $k$ where $\yb_k$ is $\tau\|\nabla f(\yb_0)\|$-stationary. In \Cref{fig:data_profile} we consider $\tau \in \{10^{-2},10^{-4}\}$.

For testing, we use the set of all CUTEst problems~\cite{Gould2014} with dimension $3 \le n \le 1000$, finite bound constraints with $l_i < u_i$, and no other additional constraints; see Appendix~\ref{appendix:cutest_problems} for a list of problem names and attributes. For problems with bounds $[\lb,\ub] \neq [0,1]^n$, we use the augmented domain warpings $\Ab(\Phib(\xb))$ where $\Ab$ is the linear map from the unit cube to the bounds $\lb,\ub$: $\Ab(\yb) := \yb\odot(\ub-\lb) + \lb$. Note that this mapping has no effect on the analysis presented in \Cref{sec:sigmoidal_connection}, since $f$ can be redefined as $f(\Ab(\cdot))$.

The set of algorithms we build data profiles for are the NLopt implementation~\cite{Johnson2019} of L-BFGS-B applied to \ref{eq:PROB}, a nonsmooth quasi-Newton method~\cite{keskar2019NQN} applied to the \ref{eq:PPM} formulation from~\cite{Galvan2021}, %
\Cref{alg:sigma_update} for updating the $\sigmab$ parameter under \ref{eq:update_rule} and using L-BFGS as a subproblem solver, and NLopt's implementation of BFGS~\cite{nocedal2006no} applied to $\ft_{\sigmab}(\xb)$ using a constant value of $\sigmab$. 

At each iteration the nonsmooth quasi-Newton method applied to PPM  is supplied with a subgradient~\cite{bagirov2014book} direction $\db_k$.
If $\xb_k$ is on the interior of $\Omegab$, the subgradient $\db_k$ is the gradient $\nabla f(\xb_k)$; if $\xb_k$ is a boundary point, then $\db_k = -(\pib(\xb_k -\nabla f(\xb_k))-\xb_k)$ is the negative projected gradient; if $\xb_k$ is exterior to $\Omegab$, then $\db_k = D\pib(\xb_k)\nabla f(\pib(\xb_k)) + \frac{\xb_k - \pib(\xb_k)}{\|\xb_k - \pib(\xb_k)\|}$, where the  second term is a normal direction to $\Omegab$ at $\pib(\xb_k)$ and where the diagonal generalized Jacobian~\cite{Clarkebook} matrix $D\pib(\xb)$ is nonzero only in components $x_i$ equal to a bound constraint. While \cite{Galvan2021} intended the variables of \ref{eq:PPM} to be $\xb_k\in\Reals^n$, these points may violate the unrelaxable constraints. Thus we opt to create the data profiles in \Cref{fig:data_profile} and~\cite{onlinesupplement} using the projected iterates $\pib(\xb_k)$ instead.

\begin{figure}[tbh!]
\includegraphics[scale=0.22]{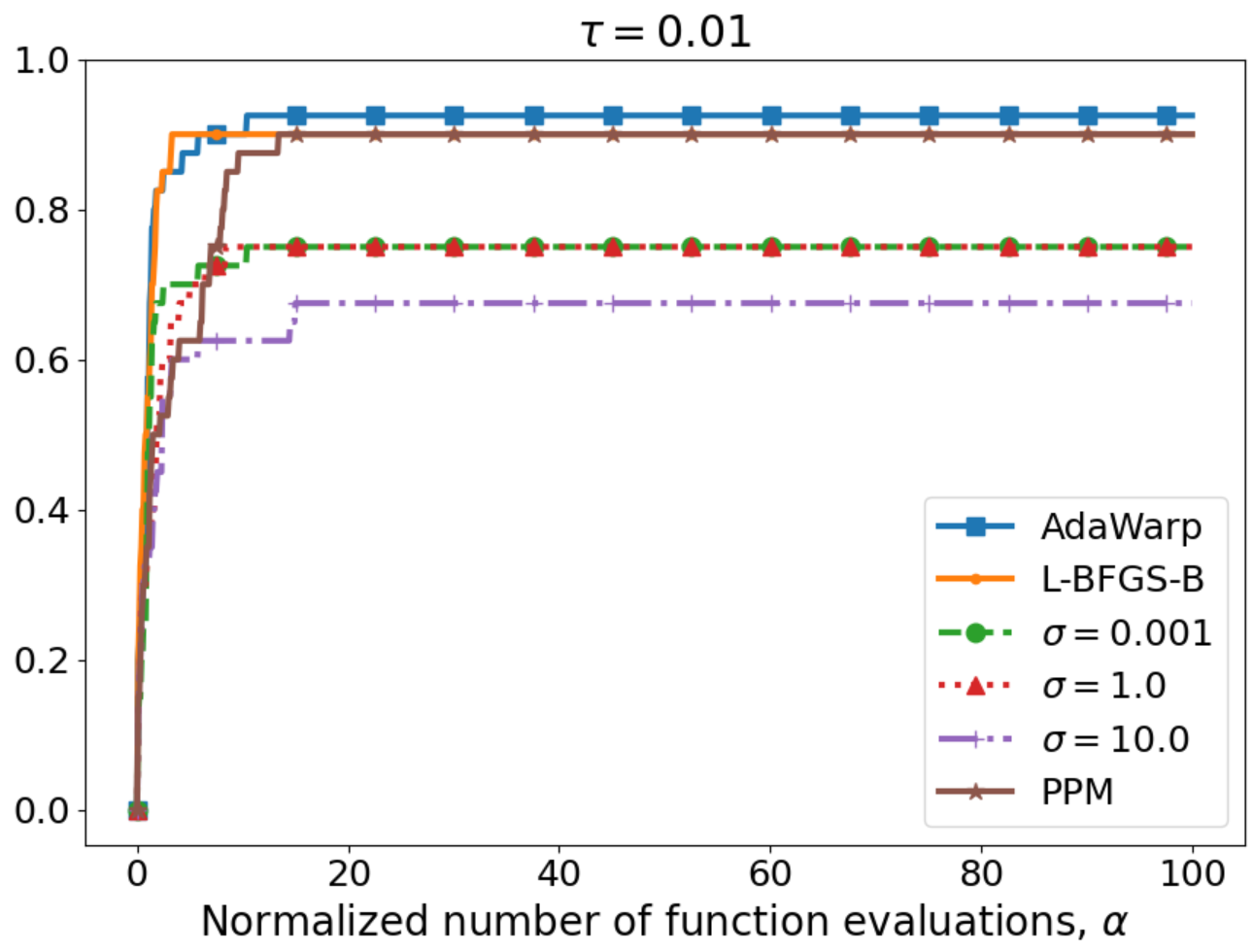}
\includegraphics[scale=0.22]{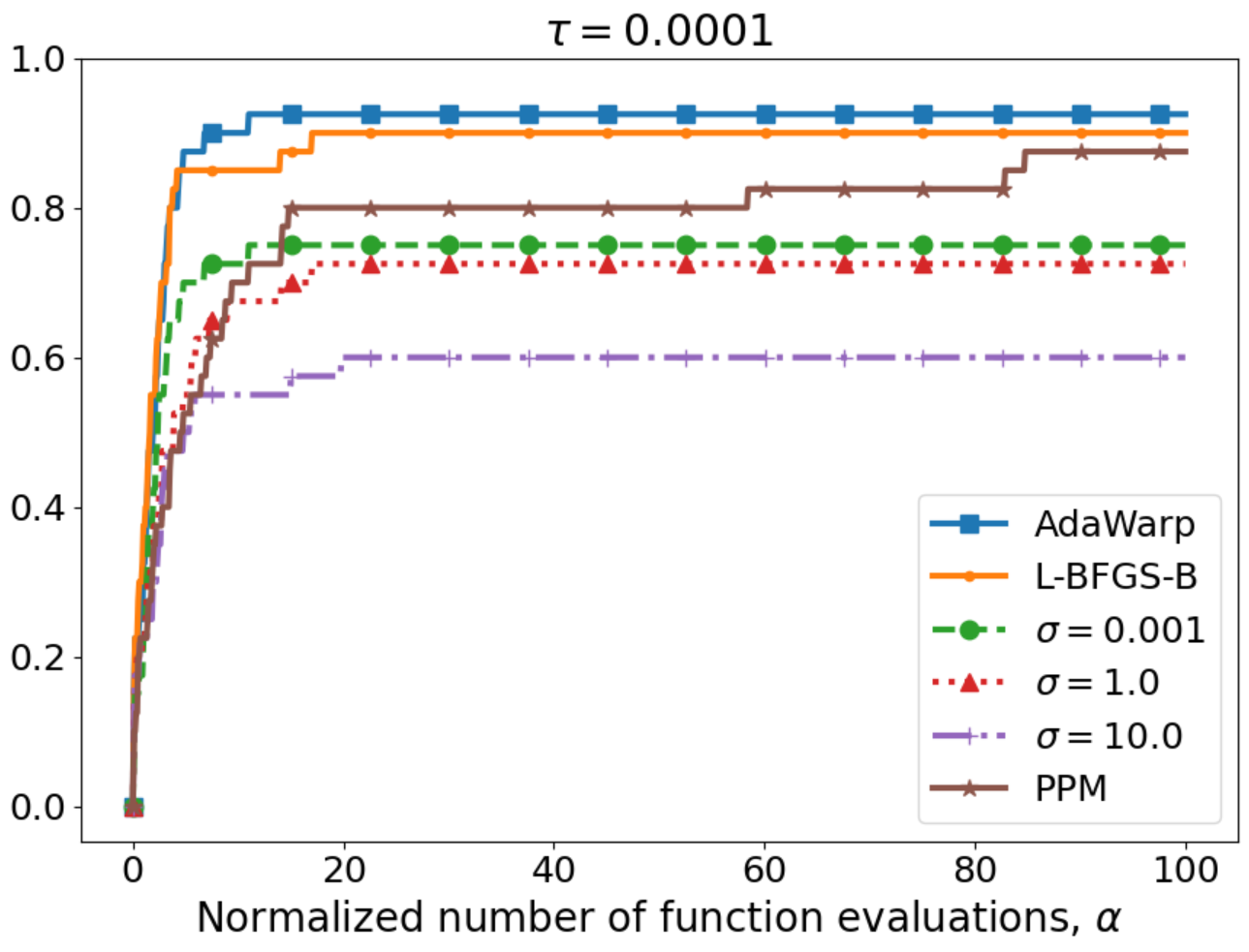}
\centering
\caption{Data profile for a relative KKT tolerance of $\tau=10^{-2}$ (left) and $\tau=10^{-4}$ (right) for L-BFGS-B, \Cref{alg:sigma_update} under \ref{eq:update_rule} and $\sigmab_0=10^{-3}$ (labeled AdaWarp), BFGS optimizing $\ft_{\sigmab}$ with a fixed $\sigmab$ (labeled $
\sigmab=0.001,1.0,10.0$), and the projection-based penalty method from~\cite{Galvan2021} (labeled PPM). When solving to a high accuracy, $10^{-4}$, L-BFGS-B and \Cref{alg:sigma_update} outperform all other methods. The projected penalty method also performs well but is not able to resolve solutions to a high accuracy as fast as the smooth methods can. Minimizing $\ft_{\sigmab}$ with a fixed value of $\sigmab$ can solve subsets of the problem rapidly but ultimately cannot resolve all problems well without the update of the $\sigmab$ parameter.}
\label{fig:data_profile}
\end{figure}

The data profiles, \Cref{fig:data_profile}, show that L-BFGS-B and \Cref{alg:sigma_update} (labeled AdaWarp) perform similarly well, solving almost all problems to a high accuracy quickly, while the optimizations of $\ft_{\sigmab}$ with a constant value of $\sigmab$ and the nonsmooth penalty approach (labeled PPM) perform somewhat worse. These results validate both the value of using a smooth domain warping and the improvement in solution resolution gained by updating $\sigmab$. The data profile with the relatively low tolerance $\tau=0.01$ shows that L-BFGS-B and \Cref{alg:sigma_update} solve nearly all problems expediently; PPM also solves almost all problems, although at a slower rate; and optimizations of $\ft_{\sigmab}$ with a constant value of $\sigmab$ solve only a fraction of the problems, although relatively quickly. When considering the data profile for a higher accuracy \Cref{fig:data_profile} (right), relative KKT tolerance of $10^{-4}$, we see that L-BFGS-B and \Cref{alg:sigma_update} can once again solve a variety of problems to a high accuracy quickly; however, %
L-BFGS-B does not solve all the problems in the set to the desired accuracy. The reason is that on a small subset of problems, the NLopt implementation of L-BFGS-B stopped prematurely because of an internal stopping tolerance, measuring the relative difference in function values, which could not be turned off. The performance of PPM drops because nonsmooth optimization problems are difficult to resolve to high orders of accuracy. For data profiles using only subsets of the problem set, see \cite{onlinesupplement}.

\section{Conclusion}
We explored the use of a domain warping to develop an analog of the classical penalty approach that applies to problems with unrelaxable finite bound constraints. The domain warping alleviates the dependence on unrelaxable constraints such that highly specialized unconstrained optimization solvers can be used to generate solutions to the unrelaxable problem. Our formulation enjoys smoothness, is easy to use, and can solve unrelaxable problems reasonably efficiently. This approach is fundamentally different from previous modeling approaches considered because it does not require a penalty term and optimizes over the interior of the domain. The domain warping detailed here applies only to problems with finite bound constraints. We conjecture that analogous warpings exist for decision sets defined by general bounded polyhedron (polytopes) and convex sets, with future research addressing settings where such warpings open up avenues of applying specialized algorithms that may otherwise ignore these types of constraints.  Such approaches are needed in sciences where problems with simple linear unrelaxable constraints, such as $x_i \le x_{i+1}$, naturally appear; examples range from ordering particle accelerator elements \cite{eldred2021RCS} to pandemic alert staging \cite{Yang2021}. We hope researchers extend this work to develop smooth modeling approaches for optimization problems with general nonlinear and convex unrelaxable constraints.

\section*{Acknowledgments}

This work was supported in part by the U.S.~Department of Energy, Office of
Science, Office of Advanced Scientific Computing Research and Office of
High-Energy Physics, Scientific Discovery through Advanced Computing (SciDAC)
Program through the FASTMath Institute and the ComPASS-4 Project under Contract
No.~DE-AC02-06CH11357.
We thank Kamil Khan for his incredible insights into nonsmooth functions and
Gokul Deez Nair for his sharp analytical insights.

\printbibliography

\appendix
\section{Mathematical Background}
\label{sec:background}
We use the Karush--Kuhn--Tucker (KKT) conditions to measure whether a point is indeed a solution, or an approximate solution, to \ref{eq:PROB}.
These measures also allow us to consider to what degree the original and reformulated problem are similar.

The KKT conditions are necessary conditions for a first-order stationary point of a constrained optimization problem. For the bound-constrained optimization problem \ref{eq:PROB} the KKT conditions \cite{nocedal2006no} are as follows.
\begin{theorem}[KKT]
\label{thm:kkt}
If $\yb^*\in\Omegab$ is a solution of {\normalfont\ref{eq:PROB}} under Assumptions {\normalfont\ref{assumption:f_L_smooth}} and {\normalfont\ref{assumption:omega_cube}}, then the KKT conditions are satisfied: There exist Lagrange multipliers $\lambdab,\mub \in \Reals^d$ such that dual feasibility holds $\lambdab,\mub \leq 0$, the stationary condition holds $\nabla f(\yb^*) + \sum_{i=1}^n \lambda_i \eb_i - \sum_{i=1}^n \mu_i \eb_i =0$, and complementary slackness holds $\lambda_i y^*_i = 0,\, \mu_i(1- y^*_i) =0$ for $i=1,\ldots,n$.
\end{theorem}

We now define approximate stationary points through a notion of approximate satisfaction of the KKT conditions, which we call $\epsilon$-stationarity.

\begin{definition}[$\epsilon$-Stationary]
\label{def:eps_stationary}
$\yb^*\in\Omegab$ is an $\epsilon$-stationary point for \ref{eq:PROB} under Assumptions \ref{assumption:f_L_smooth} and \ref{assumption:omega_cube} if the following conditions are satisfied: There exist Lagrange multipliers $\lambdab,\mub \in \Reals^n$ such that dual feasibility holds $\lambdab,\mub \leq \zerob$, the stationary condition approximately holds $\abs{\nabla f(\yb^*) + \sum \lambda_i \eb_i - \sum \mu_i \eb_i} \leq \epsilon\oneb$, and complementary slackness approximately holds $\abs{\lambdab \odot \yb^*} \leq \epsilon\oneb,\, \abs{\mub \odot (\oneb-\yb^*)} \le \epsilon\oneb$.
\end{definition}

\begin{definition}[$L$-smooth]
\label{def:L-smooth}
A function $f$ is $L$-smooth if the gradient exists and is Lipschitz continuous, that is,  for any $\xb,\yb \in \Omegab$
$$\|\nabla f(\xb) - \nabla f(\yb)\| \leq L\|\xb-\yb\|.$$
\end{definition}
A useful property of $L$-smooth functions is that they have a quadratic upper bound. 

\begin{property}[Quadratic Upper Bound]
\label{prop:quadratic_upper_bound}
If $f$ is $L$-smooth, then for any $\xb,\yb \in\Omegab$ it satisfies the following inequality~\cite{polyak2021book}:
\begin{equation*}
    f(\yb) \leq f(\xb) + \nabla f(\xb)^T(\yb-\xb) + \frac{L}{2}\|\yb - \xb\|^2 .
\end{equation*}
\end{property}

\section{Convergence and Complexity Proofs}\label{sec:appendix_proofs}

\begin{proof}[Proof of \Cref{thm:sigup_converges}]
We denote the iterates of \Cref{alg:sigma_update} as $\xb_k^*,\yb_k^*$. We use superscripts $m$ to denote the iterates of an instance of gradient descent within iteration $k$. Thus, in iteration $k$ of \Cref{alg:sigma_update}, gradient descent with a constant step size produces the sequence $\xb_k^{m+1} = \xb_k^m - \frac{1}{\tilde{L}_k}\nabla \ft_{\sigmab_k}(\xb_k^m)$ and $\yb_k^m = \Sb(\xb_k^m)$ using $\sigmab = \sigmab_k$.

We first show that at each iteration of \Cref{alg:sigma_update}, gradient descent will stop with $\|\nabla \ft_{\sigmab_k}(\xb_k^m)\|\le \delta$. Any two consecutive iterates $\xb_k^m$ and $\xb_{k}^{m+1}$ produced by gradient descent will satisfy the following sufficient decrease condition
\begin{align}
    \ft_{\sigmab_k}(\xb_k^{m+1}) &\le 
    \ft_{\sigmab_k}(\xb_k^{m}) + \nabla\ft_{\sigmab_k}(\xb_k^m)^T(\xb_k^{m+1} - \xb_k^m) + \frac{\tilde{L}_k}{2}\|\xb_k^{m+1} - \xb_k^m\|^2 \notag
    \\
    &=
    \ft_{\sigmab_k}(\xb_k^{m}) - \frac{1}{2\tilde{L}_k}\|\nabla \ft_{\sigmab_k}(\xb_k^{m})\|^2 
    \label{eq:GD_suff_decrease}
\end{align}
 as a consequence of the quadratic upper bound (see \Cref{prop:quadratic_upper_bound}).
Notice that because $f$ is bounded below on $\Omegab$, $\ft_{\sigmab_k}$ is bounded below (by $\min\{f(\xb): \xb\in \Omegab\}$). As a consequence, $\|\nabla \ft_{\sigmab_k}(\xb_k^{m})\|$ must converge to zero as $m\to\infty$, because otherwise \cref{eq:GD_suff_decrease} would contradict this lower bound. This implies that gradient descent will terminate after finding an iterate satisfying the stopping tolerance $\|\nabla \ft_{\sigmab_k}(\xb_{k}^{N_k})\| \le \delta$ after $N_k$ iterations.

Having shown that gradient descent converges for any iteration of \Cref{alg:sigma_update}, we now show that for sufficiently large $k$ the decrease in objective value achieved over any gradient descent step  depends only on $\sigmab_k$ through the domain warping. To this end we must remove the effect of the Lipschitz constant $\tilde{L}_k$ in \cref{eq:GD_suff_decrease}, which is dependent on $\sigmab_k$. \Cref{thm:lipschitz_consant_bound} expresses the Lipschitz constant as $\tilde{L}_k = \frac{1}{2}(\sigma_{\max,k}^2\hat{L} + \sigma_{\max,k}L)$, where $L,\hat{L}$ are 
constants and 
$\sigma_{\max,k} = \max_j\{\sigma_{k,j}\}$. 
For the remainder of the proof suppose that $k$ is large enough that $\sigmab_k\ge \oneb$. Then we can upper bound the Lipschitz constant via $\tilde{L}_k = \frac{1}{2}(\sigma_{\max,k}^2\hat{L} + \sigma_{\max,k}L) \le\frac{1}{2} \sigma_{\max,k}^2(\hat{L} + L) = \frac{1}{2}\sigma_{\max,k}^2C$ for $C = \hat{L} + L$.

We rewrite the sufficient decrease condition \cref{eq:GD_suff_decrease} in terms of $\yb_k^m$ and uncover the presence of $\sigmab_k$. To do so, we expand the derivative terms, lower bound $\|\sigmab\|$ in terms of $\sigma_{\min,k} =\min_j\{\sigma_{k,j}\}$, and use the Lipschitz constant bound
\begin{align*}
    f(\yb_k^{m+1}) &\le  f(\yb_k^{m}) - \frac{1}{\sigma_{\max,k}^2\hat{L} + \sigma_{\max,k}L }\|\sigmab_k\odot\yb_k^m\odot(\oneb-\yb_k^m)\odot\nabla f(\yb_k^m)\|^2
    \\
    &\le f(\yb_k^{m}) - \frac{\sigma_{\min,k}^2}{\sigma_{\max,k}^2C}\|\yb_k^m\odot(\oneb-\yb_k^m)\odot\nabla f(\yb_k^m)\|^2 .
\end{align*}
With the lower bound of $\kappa$ on $\sigma_{\min,k}/\sigma_{\max,k}$ we arrive at a sufficient decrease condition that is affected by $\sigmab_k$ only through the domain warping:
\begin{align}
    f(\yb_k^{m+1}) &\le f(\yb_k^{m}) - \frac{\kappa^2}{C}\|\yb_k^m\odot(\oneb-\yb_k^m)\odot\nabla f(\yb_k^m)\|^2.
    \label{eq:GD_suff_decrease_sigmaless}
\end{align}
This bound not only holds between iterates within the same step $k$ of \Cref{alg:sigma_update} but also can connect steps of \Cref{alg:sigma_update}. We emphasize that the last iterate (in the original domain) $\yb_k^{N_k} = \Sb(\xb_k^{N_k})$ found by gradient descent will be equal to the iterate $\yb_k^*$ of \Cref{alg:sigma_update} as well as the first iterate of the gradient descent after $\sigmab$ is updated, that is, 
\begin{equation}\label{eq:consecutive_ys}
\yb_k^{N_k} = \yb_k^* = \yb_{k+1}^0. 
\end{equation}
(The same cannot be said for the warped counterparts, because $\xb_k^{N_k}$ and $\xb_{k+1}^0$ are not warped under the same value of $\sigmab$.)
By \cref{eq:consecutive_ys},
\begin{align*}
    f(\yb_{k+1}^{1}) 
    &\le f(\yb_{k+1}^{0}) - \frac{\kappa^2}{ C}\|\yb_{k+1}^{0}\odot(\oneb-\yb_{k+1}^{0})\odot\nabla f(\yb_{k+1}^{0})\|^2
    \\
    &= f(\yb_{k}^{N_k}) - \frac{\kappa^2}{ C}\|\yb_{k}^{N_k}\odot(\oneb-\yb_{k}^{N_k})\odot\nabla f(\yb_{k}^{N_k})\|^2 .
\end{align*}
Because consecutive iterates generated by gradient descent and \Cref{alg:sigma_update} satisfy this quadratic growth condition, we relabel the sequence of points $\{\yb_k^m\}$ as $\{\hat{\yb}_l\}$.
We exclude the points $\yb_k^{N_k}$ from the sequence $\{\hat{\yb}_l\}$ so they do not appear twice (since they are equal to $\yb_{k+1}^0$). Similarly we make the sequences $\{\hat{\xb}_l\}$, $\{\sigmab_l\}$, and $\{\alpha_l\}$. Thus, for sufficiently large $l$ such that $\sigmab_l \ge \oneb$, consecutive iterates of the sequence $\{\hat{\yb}_l\}$ satisfy the quadratic growth condition
\begin{align}\label{eq:term}
    f(\hat{\yb}_{l+1}) &\le f(\hat{\yb}_l) - \frac{\kappa^2}{2C}\|\hat{\yb}_l\odot(\oneb-\hat{\yb}_l)\odot\nabla f(\hat{\yb}_l)\|^2 .
\end{align}
The objective function $f$ is bounded below on $\Omegab$, which implies that the term $\hat{\yb}_l\odot(\oneb-\hat{\yb}_l)\odot\nabla f(\hat{\yb}_l)$ from \cref{eq:term} must converge to zero. Equivalently, for each component $j$, either $\hat{y}_{l,j}$ converges to the boundary or $\partial_j f(\hat{\yb}_l)\to 0$. While this decrease condition implies that the limit points of gradient descent will be stationary for $\ft_{\sigmab}$, it does not imply that the KKT conditions for \ref{eq:PROB} will be satisfied in the limit. For that, we will appeal to properties of the gradient descent step sequence. 

Because $\{\hat\yb_l\}$ is a bounded sequence, the Bolzano--Weierstrass theorem guarantees that there exists a convergent subsequence, $\hat{\yb}_{l_i} \to \zb \in \Omegab$. We will show that any limit point $\zb$ of $\{\hat{\yb}_{l}\}$ must be a KKT point.

If $j$ is a component such that $\partial_j f(\hat{\yb}_{l_i}) \to 0$, then the KKT conditions are satisfied in component $j$ at $\zb$. So, let $j$ be a component such that $\partial_j f(\hat{\yb}_{l_i})$ does not converge to zero. This implies that $\hat\yb_{l_i}$ converges to the boundary, namely, $z_j = 0$ or $1$. Without loss of generality, assume $z_j = 1$. For sake of contradiction assume $\partial_j f(\zb)>0$.

By continuity of $\partial_j f$, there exists a $\beta >0$ such that if $\|\hat\yb_{l_i} - \zb\|_{\infty} \le \beta$, then $\partial_j f(\hat\yb_{l_i}) >0$. Furthermore, by the convergence of $\hat\yb_{l_i}\to\zb$ there exists an $M$ such that for all $l_i> M$, $\|\hat\yb_{l_i} - \zb\|\le \beta$ and consequently $\partial_j f(\hat\yb_{l_i}) >0$. When $\partial_j f(\hat\yb_{l_i}) >0$, the gradient descent step implies that 
\begin{equation*}
    \hat{x}_{l_{i+1},j} < \hat{x}_{l_i,j} - \alpha_{l_i} \sigma_{l_i,j} \hat{y}_{l_i,j}(1-\hat{y}_{l_i,j})\partial_j f(\hat{\yb}_{l_i}) < \hat{x}_{l_i,j} .
\end{equation*}
Monotonicity of the sigmoidal warping ensures that $\hat{y}_{l_{i+1},j} < \hat{y}_{l_i,j}$ and hence that $\hat{y}_{l_{i+1},j}$ is further from $z_j$ than $\hat{y}_{l_{i+1},j}$; that is, $z_j - \hat{y}_{l_{i+1},j} > z_j -\hat{y}_{l_i,j}$.

This process continues inductively so long as $\partial_jf(\hat{\yb}_{l_i})>0$, which holds for all ${l_i}>M$; that is, for any ${l_i}> M$,  $z_j - \hat{y}_{l_{i+1},j} > z_j -\hat{y}_{l_i,j}$. Thus, the distance from $\hat{y}_{l_i,j}$ to $z_j$ is bounded below by $\abs{z_j - \hat{y}_{l_i,j}} > \abs{z_j - \hat{y}_{M,j}}$, which implies that $\{\hat{y}_{l_i,j}\}$ does not converge to $z_j$. This is a contradiction, and so it must hold that $\partial_j f(\zb) \le 0$. 

Now that we have shown that the derivatives at the boundary point $\zb$ have the appropriate sign, the KKT conditions holds: There exists a dual feasible $\lambda_j \geq 0$ such that the stationary condition holds, namely,  $\partial_j f(\zb)+ \lambda_j = 0$, and complementary slackness holds (because $z_j$ lies on the boundary of $\Omegab$). 
\end{proof}

\begin{proof}[Proof of \Cref{thm:sigup_num_iter}]
  To prove this claim, we will show that $\sigmab_k$ increases sufficiently such that the bounds in \Cref{thm:stationary_condition} imply $\epsilon$-stationarity.

As a first case, suppose there exist components $j$ such that the sequence $\{\yb_{k,j}^*\}$ is bounded away from the boundary of $\Omegab$; in other words,  for all $k>0$  the sequence lies in $y_{k,j}^* \in [\nu,1-\nu]$, where $\nu\in (0,1)$. For $k > 0$, \ref{eq:update_rule} sets the parameter $\sigmab_{k}$ to have components $\sigma_{k,j} = \frac{\gamma^{k}}{\prod_{l=0}^{k-1}\sqrt{\eta_{l,j}}}$. Notice that we have assumed here that $\sigmab_{k+1} = \frac{\gamma}{\sqrt{\etab_k}}\odot\sigmab_k$. When iterating to a finite tolerance, $\sigmab_k$ does not need to diverge, and so a value of $\kappa$ always exists such that $\min_{u,v}\{\sigma_{k,u}/\sigma_{k,v}\}> \kappa$ for all $k$.
Since $\eta_{k,j}$ is the minimum distance from $y_{k,j}^*$ to the boundary, it is bounded above $\eta_{k,j}\leq 1/2$ for all $k$. By \Cref{thm:stationary_condition} and using $\abs{\partial_j \ft_{\sigmab}(\yb_k^*)} <\delta$ as guaranteed by \Cref{alg:sigma_update}, the value of the partial derivatives is bounded by
\begin{align*}
    \abs{\partial_j f(\yb_k^*)} &\leq \frac{\delta}{\sigma_{k,j} y_{k,j}^*\left(1-y_{k,j}^*\right)}
    \leq\frac{\delta\prod_{l=0}^{k-1}\sqrt{\eta_{l,j}}}{\gamma^k\nu(1-\nu)}
    \leq\frac{\delta}{(\sqrt{2}\gamma)^k\nu(1-\nu)}.
\end{align*}
This bound ensures that for all $k \ge \log(\frac{\delta}{\epsilon \nu(1-\nu)})/\log(\sqrt{2}\gamma)$ steps, $y_{k}^*$ is $\epsilon$-stationary in component $j$, that is,  $\abs{\partial_j f(\yb^*_{k})} \le \epsilon$. If $\{y^*_{k,j}\}$ is uniformly bounded away from the boundary in all components, then $\{\yb_k^*\}$ will contain an $\epsilon$-stationary point within
\begin{equation*}
    N = \log\left(\frac{\delta}{\epsilon \nu (1-\nu )}\right)/\log(\sqrt{2}\gamma)
\end{equation*}
iterations.

Now consider the set of components $j$ such that the sequence $\{y_{k,j}^*\}$ is not contained within a compact set $[\nu,1-\nu]$ for $\nu\in(0,1)$. As discussed in the proof of \Cref{thm:sigup_converges}, this implies that a subsequence of $\{y_{k,j}^*\}$ converges to a component $z^*_j$of a stationary point such that $z^*_j$ is on the boundary of $\Omegab$. Using \Cref{assumption:same_side} and  \Cref{thm:stationary_condition}, we find that the error in the stationary condition with respect to the component $z^*_j$ with Lagrange multiplier $\lambda^*_j$ is bounded by
\begin{align*}
    \abs{\partial_j f(\yb_k^*) +\lambda_j^*}
&\leq \dfrac{L_j\delta}{\abs{\partial_j f(\yb_k^*)}\sigma_{k,j} \Delta_{k,j}}
=\dfrac{L_j\delta\prod_{l=0}^{k-1}\sqrt{\eta_{l,j}}}{\abs{\partial_j f(\yb_k^*)}\gamma^k \Delta_{k,j}}
\leq \dfrac{L_j\delta}{\abs{\partial_j f(\yb_k^*)}(\sqrt{2}\gamma)^k\xi}.
\end{align*}
This bound ensures that for all $k \ge \log(\frac{L_j\delta}{\xi \epsilon^2})/\log(\sqrt{2}\gamma)$ either $\abs{\partial_j f(\yb_k^*) +\lambda_j^*} \le \epsilon$ or $\abs{\partial_jf(\yb_k^*)}\le \epsilon$. In either case, $\epsilon$-stationarity is satisfied in component $j$.

Moreover, within
\begin{equation*}
    N = \max\left\{\log\left(\frac{\delta}{\epsilon \nu(1-\nu)}\right)/\log(\sqrt{2}\gamma), \log\left(\frac{\overline{L}\delta}{\xi \epsilon^2} \right)/\log(\sqrt{2}\gamma)\right\}
\end{equation*}
iterations, $\{\yb_{k}^*\}$ contains an $\epsilon$-stationary point. If no components of the sequence $\{\yb_k^*\}$ stay uniformly bounded away from the boundary, then the bound simplifies to
\begin{equation*}
    N = \log\left(\frac{\overline{L}\delta}{\xi \epsilon^2} \right)/\log(\sqrt{2}\gamma).
\end{equation*}
\end{proof}

\section{CUTEst Problems}
\label{appendix:cutest_problems}

For testing, we use the set of all CUTEst problems~\cite{Gould2014} with dimension $3 \le n \le 1000$, finite bound constraints (not including equality constraints), and no other additional constraints. The problems are listed in \Cref{table:cutest_problems}.

\begin{table}[h!]
\caption{The 40 bound-constrained CUTEst problems used to create the data profiles in \Cref{fig:data_profile}. The columns show the problem name, problem type (Q for quadratic, S for sum of squares, O for other), problem dimension $n$, and approximate number of active constraints at the local optima. The number of active constraints was determined by running L-BFGS-B to a gradient tolerance of $10^{-6}$ to find a local optimum; then the number of activities was computed as the number of components of the optimum with a distance to a boundary less than $0.1\%$ of the componentwise domain width.}
\label{table:cutest_problems}

\centering
\footnotesize
\begin{tabular}{|c|ccc|}
\hline
 Problem &Type &$n$ &Active Constraints \\
 \hline
BQPGABIM &Q &46 &13\\BQPGASIM &Q &50 &14\\CHEBYQAD &S &100 &0\\DEVGLA1B &S &4 &0\\DEVGLA2B &S &5 &2\\DGOSPEC &O &3 &2\\DIAGIQB &Q &1000 &573\\DIAGIQE &Q &1000 &1000\\DIAGIQT &Q &1000 &7\\DIAGNQB &Q &1000 &735\\DIAGNQE &Q &1000 &676\\DIAGNQT &Q &1000 &496\\DIAGPQB &Q &1000 &0\\DIAGPQE &Q &1000 &0\\DIAGPQT &Q &1000 &0\\FBRAIN2LS &S &4 &1\\GENROSEB &S &500 &494\\HADAMALS &O &380 &19\\HART6 &O &6 &0\\HS110 &S &10 &0\\HS25 &S &3 &1\\HS38 &O &4 &0\\HS45 &O &5 &5\\LEVYMONT &S &100 &0\\LEVYMONT10 &S &10 &0\\LEVYMONT6 &S &3 &0\\LEVYMONT7 &S &4 &0\\LEVYMONT8 &S &5 &0\\LEVYMONT9 &S &8 &0\\MAXLIKA &O &8 &3\\POWELLBC &O &1000 &32\\POWERSUMB &S &4 &0\\PROBPENL &O &500 &0\\QINGB &S &5 &0\\S368 &O &8 &2\\SANTALS &S &21 &0\\SINEALI &O &1000 &0\\SPECAN &S &9 &0\\STRTCHDVB &S &10 &0\\TRIGON1B &S &10 &0\\
\hline
\end{tabular}
\end{table}

A few of these problems (DEVGLA2B, GENROSEB, HS25, HS45, MAXLIKA, POWELLBC) had nominal starting points with one or more components outside or on the bound constraints. In this case, we translated the components that were not interior to the bounds orthogonally into the feasible region by $0.1\%$ of the bound width. The affine mapping $(\ub-\lb)\odot \yb + \lb$ was used to map points on the unit cube to the bound-constrained domains $[\lb,\ub]$. This totals 40 problems, of which we estimate, through the use of a bound-constrained solver, that 18 have solutions on the boundary of the domain. The distribution of problem dimensions is given in \Cref{table:problem_dimensions}.

\begin{table}[h!]
\caption{Distribution of problem dimensions $n$.}
\label{table:problem_dimensions}

\centering
\small
\begin{tabular}{ c|p{.1cm}p{.1cm}p{.1cm}p{.1cm}p{.1cm}p{.1cm}p{.1cm}p{.1cm}p{.1cm}p{.2cm}p{.2cm}p{.2cm}p{.2cm}c } 
 $n$ &3 &4 &5 &6 &8 &9 &10 &21 &50 &100 &380 &500 &1000 & \\ 
 \hline
 Number of problems &3 &5 &4 &1 &3 &1 &5 &1 &1 &2 &1 &2 &11 \\ 
\end{tabular}
\end{table}

\section{Domain Warpings for Other Decision Sets}
\label{appendix:warpings_for_other_sets}
In this appendix we discuss variations of the sigmoidal warping that map onto other decision sets. \Cref{table:domain_warps_other_sets} shows domain warpings $\Mb:\Reals^n\to\Gammab$, similar to the sigmoidal warping, for decision sets $\Gammab$ defined by nonnegativity constraints, simplexes, and the unit cube. For decision sets that are a Cartesian product of the prior, the warping can be defined as a Cartesian product as well. Furthermore, if a decision set $\hat{\Gammab}$ can be defined as a smooth invertible map $\Tb(\yb):\Gammab\to\hat{\Gammab}$, then the domain warping can be defined as $\Phib(\xb) = \Tb(\Mb(\xb))$.

\begin{table}[h!]
\caption{Decisions sets $\Gammab$ and domain warpings $\Mb:\Reals^n\to \Gammab$ from the unconstrained domain to the decision set, with warping parameter $\sigmab\in\Reals^n_{++}$. From left to right the decision sets are nonnegativity constraints, nonempty bounded simplexes parameterized by $\ab\in\Reals^n_{++}$ and $b\in\Reals_{++}$, and the unit cube. The theory from \Cref{sec:sigmoidal_connection} should be extendable to these domain warpings with minor modifications.}
\label{table:domain_warps_other_sets}

\def\arraystretch{2}
\centering
\begin{tabular}{|c|c|c|c|}
\hline
Decision Set & $\yb \geq \zerob$  &$\yb\geq \zerob$, $\ab^T\yb \le b$ &$\zerob\le \yb \le \oneb$
\\
\hline
Warping &$e^{\sigmab\odot\xb}$ &$\frac{be^{\sigmab\odot\xb}}{\oneb+\ab^Te^{\sigmab\xb}}$ &$\frac{\oneb}{\oneb+e^{-\sigmab\odot\xb}}$
\\
\hline
\end{tabular}
\end{table}

\section{Algorithmic Considerations}
\label{app:alg_recs}

We now discuss algorithmic considerations that will develop insights through the structure of $\ft_{\sigmab}$. First, we investigate the role that $\sigmab$ plays in reshaping gradient steps. Next, we describe a steepest-descent method that adapts to the structure of the sigmoidal warping $\Sb$ (steepest-descent \cite[Ch.9]{Boyd2004} refers to methods that take a step in the steepest direction with respect to a not-necessarily Euclidean norm). 

\subsection{Effect of the sigmoidal warping on gradient steps}
\label{sec:sigmoidal_gradient_steps}
Many first-order optimization algorithms find the next iterate by minimizing a quadratic model of the objective, such as
\begin{equation}
    m_k(\xb) = \ft_{\sigmab}(\xb_k) + \nabla \ft_{\sigmab}(\xb_k)^T(\xb-\xb_k) + \frac{1}{2\alpha}\|\xb - \xb_k\|^2.
    \label{eq:quadratic_GD_model}
\end{equation}
The model in \cref{eq:quadratic_GD_model} yields a step along the negative gradient direction via the rule $\xb_{k+1} = \xb_k - \alpha \nabla \ft_{\sigmab}(\xb_k)$.
While this routine may make consistent improvement with respect to $\ft_{\sigmab}$, the corresponding sequence $\yb_k = \Sb(\xb_k)$ may make poor improvement toward the optima $\yb^*$ in the original domain. The sequence $\yb_k$ can be approximated to first order by 
\begin{align*}
    \yb_{k+1} &\approx \yb_k - \alpha\sigmab\odot \yb_k\odot(\oneb-\yb_k)\odot\nabla \ft_{\sigmab}(\xb_k)
    \\
    &= \yb_k - \alpha\sigmab^2\odot \yb_k^2\odot(\oneb-\yb_k)^2\odot\nabla f(\yb_k).
\end{align*}
If the product $\sigma_iy_{k,i}(1-y_{k,i})$ is small for some component $i$, then the component $y_{k+1,i}$ will not change substantially through the iteration, even when $\abs{\partial_i f(\yb_k)}$ is large. This naturally occurs when $y_{k,i}$ approaches the boundary as a byproduct of the sigmoidal warping, which forces all elements $\yb_k$ to be within $\Omegab$. Furthermore, we can see how this affects the improvement in the objective throughout this iteration through, once again, a first-order Taylor expansion:
\begin{align*}
    f(\yb_{k+1}) &\approx f(\yb_k) - \alpha[\sigmab\odot \yb_k\odot(\oneb-\yb_k)\odot\nabla f(\yb_k)]^T\nabla \ft_{\sigmab}(\xb_k)
    \\
    &= f(\yb_k) - \alpha \nabla f(\yb_k)^T\sigmab^2\odot \yb_k^2\odot(\oneb-\yb_k)^2\odot\nabla f(\yb_k).
\end{align*}
This approximation suggests that if a single entry $y_{k,i}$ is near the boundary and another entry $y_{k,j}$ is sufficiently interior and $\sigmab$ is not set appropriately, then the gradient descent step would be inefficient, since the step direction $\db_k = \sigmab\odot \yb_k\odot(\oneb-\yb_k)\odot\nabla f(\yb_k)$ would be nearly orthogonal to $\nabla f(\yb_k)$. Setting $\sigmab$ with the sequence $\sigmab_k = (\yb_k\odot(\oneb-\yb_k))^{-1}$ will optimally align the gradient $\nabla \ft_{\sigmab}(\yb_k)$ with $\nabla f(\yb_k)$ and could improve the convergence rate of the method with respect to $f$. However, particularly when the sequence $\yb_k$ is approaching the boundary, this could significantly reduce the step size $\alpha$, since the step would be in the direction $\nabla f(\yb_k)$, which is likely to be nearly orthogonal to the boundary when near a boundary optima. On the other hand, if $\sigmab$ was a fixed constant, the gradient descent step direction would be naturally rotated to point along the boundary.
In the following section we formalize this intuition and package it into a steepest-descent routine.

\subsection{Steepest descent with sigmoidal norm}
\label{sec:steepest_descent}

\begin{figure}[tbh!]
\includegraphics[scale=0.5]{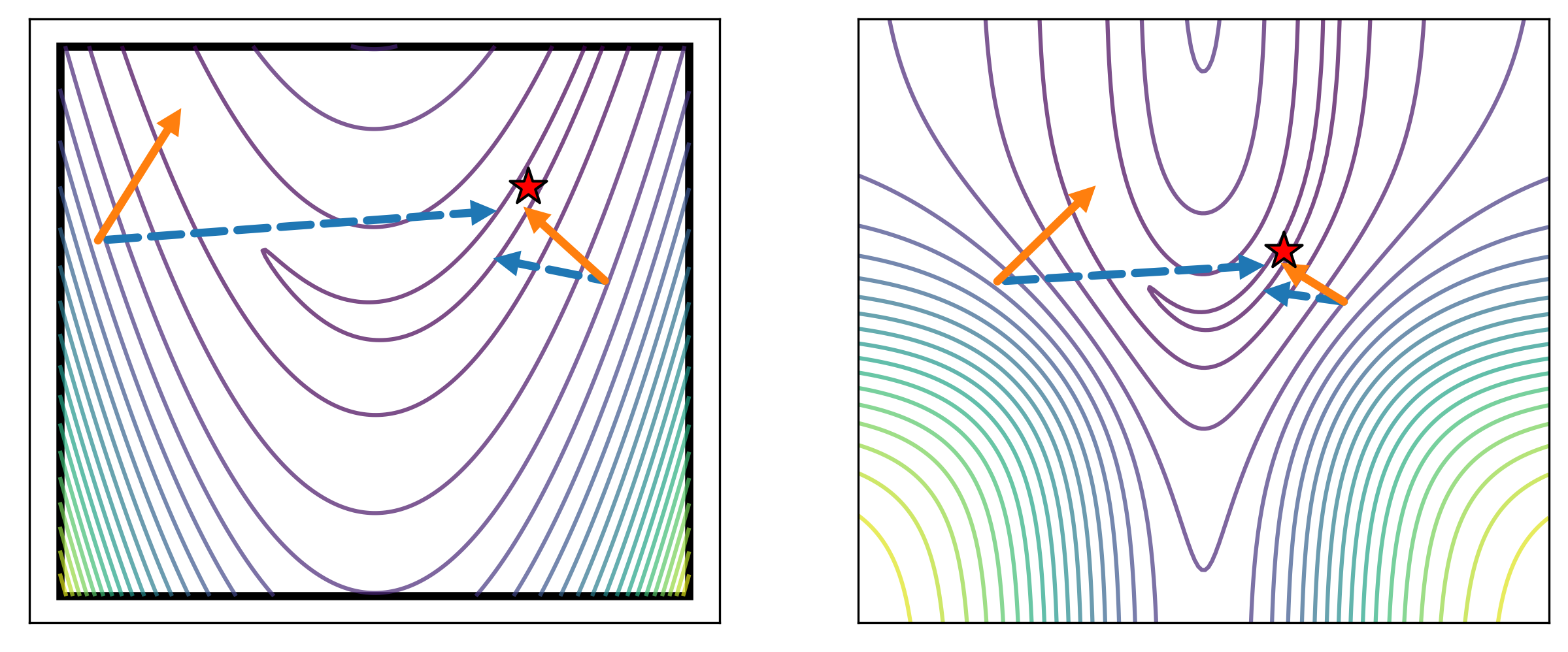}
\centering
\caption{Steepest-descent steps from \cref{eq:steepest_descent_step} (dashed blue arrows) and gradient-descent steps (solid orange arrows) on the merit function $\ft_{\sigmab}$ (right), and the resultant step mapped under the sigmoidal warping $\Sb$ to the constrained domain (left). $f$ is the Rosenbrock function. The steepest descent step from the left point, in both figures, overcomes the warping of the sigmoidal function that occurs near the boundary and takes a more efficient step toward the minima (red star) than the gradient-descent step from the same point. However, taking a steepest-descent step on $\ft_{\sigmab}$ is akin to taking a gradient-descent step on $f$, which may be  inefficient for poorly conditioned functions such as the Rosenbrock function; the gradient-descent step  on $\ft_{\sigmab}$ from the right point, in either figure, is more efficient than the corresponding steepest-descent step. The step lengths were computed by using a line search.}
\label{fig:steepest_descent}
\end{figure}

While the optimization of $\ft_{\sigmab}$ proceeds over points $\xb$ in the unconstrained domain, the success of the optimization is measured in terms of distance between $\Sb(\xb)$ and the optima $\yb^*$. This suggests that the Euclidean distance is not an appropriate distance measure in the unconstrained domain and that a better distance measure should capture the curvature of the manifold generated by $\Sb$. 
For two points $\xb,\xb'\in\Reals^n$ near one another in the unconstrained domain, the distance between $\yb = \Sb(\xb)$ and $\yb' :=\Sb(\xb')$ is approximately
\begin{align*}
    \left\|  \yb' - \yb\right\| &\approx \left\|  \Sb(\xb') - ( \Sb(\xb') +J_{\sigmab}(\xb')(\xb-\xb') )\right\| \\&= \|J_{\sigmab}(\xb')(\xb-\xb')\|
    \\
    &=:\|\xb-\xb'\|_{\sigmab}.
\end{align*}
Since this \say{sigmoidal norm} more accurately captures the distance of interest, we can improve our optimization by using a steepest descent method equipped with the sigmoidal norm $\|\cdot\|_{\sigmab}$ in place of the conventional gradient descent, which relies on the Euclidean norm. We use the terminology steepest descent, similar to \cite{Boyd2004}, to denote a method that takes a step in the steepest direction with respect to a (not-necessarily Euclidean) norm. For a review of steepest descent versus gradient descent see, for example, \cite[Ch 9]{Boyd2004}. Steepest descent with the sigmoidal norm chooses the iterate $\xb_{k+1}$ by minimizing the quadratic model:
\begin{equation*}
    \xb_{k+1} = \arg\min_{\xb}\ft_{\sigmab}(\xb_k) + \nabla \ft_{\sigmab}(\xb_k)^T(\xb-\xb_k) + \frac{1}{2\alpha_k}\|\xb - \xb_k\|_{\sigmab}^2,
\end{equation*}
which yields the following sequence:
\begin{align}
    \xb_{k+1} &= \xb_k - \alpha_k(J_{\sigmab}(\xb_k)J_{\sigmab}(\xb_k))^{-1}\nabla\ft_{\sigmab}(\xb_k) \notag
    \\
    &=\xb_k - \alpha_k\diag(\sigmab\odot\yb_k\odot(\oneb-\yb_k))^{-1}\nabla f(\yb_k).
    \label{eq:steepest_descent_step}
\end{align}
See \Cref{fig:steepest_descent} for a comparison of the steepest-descent step direction with the gradient-descent step. To a first-order approximation, the improvement in $f$ is akin to taking a gradient step in the original domain:
\begin{align*}
    f(\yb_{k+1}) &\approx f(\yb_k) - \alpha \nabla f(\yb_k)^T\nabla f(\yb_k).
\end{align*}
As discussed in \Cref{sec:sigmoidal_gradient_steps}, however, the step direction may be nearly orthogonal to the gradient $\nabla \ft_{\sigmab}$, resulting in poor improvement. We can measure the orthogonality through the inner product $\nabla f(\yb_k)^T\diag(\sigmab\odot\yb_k\odot(\oneb-\yb_k))^{-1}\nabla f(\yb_k)$, which may be near zero when some, but not all, indexes of $\yb_k$ are near the boundary of $\Omegab$. A compromise to make this method practical is to use it as a hybrid method with gradient descent, by taking traditional gradient steps when the steepest-descent step direction is highly orthogonal to $\nabla \ft_{\sigmab}$ and taking steepest descent steps otherwise.

\framebox{\parbox{.90\linewidth}{\scriptsize The submitted manuscript has been created by
        UChicago Argonne, LLC, Operator of Argonne National Laboratory (``Argonne'').
        Argonne, a U.S.\ Department of Energy Office of Science laboratory, is operated
        under Contract No.\ DE-AC02-06CH11357.  The U.S.\ Government retains for itself,
        and others acting on its behalf, a paid-up nonexclusive, irrevocable worldwide
        license in said article to reproduce, prepare derivative works, distribute
        copies to the public, and perform publicly and display publicly, by or on
        behalf of the Government.  The Department of Energy will provide public access
        to these results of federally sponsored research in accordance with the DOE
        Public Access Plan \url{http://energy.gov/downloads/doe-public-access-plan}.}}

\end{document}